\documentclass{article}
\usepackage{latexsym}
\usepackage{leftidx}
\usepackage{slashed}
\usepackage{amsmath,amsthm, pdfpages}
\usepackage{enumerate}
\usepackage{amssymb}
\usepackage{upgreek}

\usepackage[margin=1.3in,dvips]{geometry}

\newcommand{\vol}{\textnormal{vol}}

\def\f12{\frac 1 2}

\def\ga{\gamma}

\def\ep{\epsilon}

\def\cL{\mathcal{L}}

\def\Lb{\underline{L}}

\def\pa{\partial}
\def\les{\lesssim}

\def\cD{\mathcal{D}}
\def\R{\mathbb{R}}

\newcommand{\nabb}{\mbox{$\nabla \mkern-13mu /$\,}}

\newtheorem{thm}{Theorem}

\newtheorem{Prop}{Proposition}[section]

\newtheorem{Lem}{Lemma}[section]
\newtheorem{cor}{Corollary}
\newtheorem{Cor}{Corollary}[section]

\newtheorem{Remark}{Remark}[section]

\begin{document}

\title{On the global behaviors for defocusing semilinear wave equations in $\mathbb{R}^{1+2}$}

\author{Dongyi Wei \and Shiwu Yang}

\AtEndDocument{\bigskip{\footnotesize%
  \addvspace{\medskipamount}
  \textsc{School of Mathematical Sciences, Peking University, Beijing, China} \par
  \textit{E-mail address}: \texttt{jnwdyi@pku.edu.cn} \par

  \addvspace{\medskipamount}
  \textsc{Beijing International Center for Mathematical Research, Peking University, Beijing, China} \par
  \textit{E-mail address}: \texttt{shiwuyang@math.pku.edu.cn}
  }}

\date{}

\maketitle

\begin{abstract}
In this paper, we study the asymptotic decay properties for defocusing semilinear wave equations in $\mathbb{R}^{1+2}$ with pure power nonlinearity. By applying new vector fields to null hyperplane, we derive improved time decay of the potential energy, with a consequence that the solution scatters both in the critical Sobolev space and energy space for all $p>1+\sqrt{8}$. Moreover combined with  Br\'{e}zis-Gallouet-Wainger type of logarithmic Sobolev embedding, we show that the solution decays pointwise with sharp rate $t^{-\frac{1}{2}}$ when $p>\frac{11}{3}$ and with rate $t^{ -\frac{p-1}{8}+\ep }$ for all $1<p\leq \frac{11}{3}$. This in particular implies that the solution scatters in energy space when $p>2\sqrt{5}-1$.

\end{abstract}

\section{Introduction}
Consider the Cauchy problem to the defocusing semilinear wave equation
\begin{equation}
  \label{eq:NLW:semi:2d}
  \Box\phi=-\pa_t^2\phi+\Delta\phi= |\phi|^{p-1}\phi,\quad \phi(0, x)=\phi_0(x),\quad \pa_t\phi(0, x)=\phi_1(x)
\end{equation}
  in $\mathbb{R}^{1+2}$ with $1<p<\infty$. It is well known that this equation is energy subcritical and is local well-posed in energy space, see for example \cite{velo85:global:sol:NLW}. The energy conservation then immediately implies that the solution is indeed globally in time but without any asymptotic dynamics.

Due to the repulsive nature of the equation, it is expected that the solution should decay in some sense. It has been shown in the early work \cite{Pecher82:NLW:2d} of Glassey-Pecher that the solution decays as fast as linear solutions in the pointwise sense for the super-conformal case $p> 5$. The solution still decays as long as $1+\sqrt{8}<p\leq 5$ but with slower decay rate in time. As a consequence of these pointwise decay estimates, they also demonstrated that the solution scatters in energy space when $p>4.15$.

While the above scattering result (asymptotic completeness) is measured in energy space and heavily relies on the smoothness of the data, Ginibre-Velo in  \cite{Velo87:decay:NLW} established a complete scattering theory (asymptotic completeness and existence of wave operator) with data in conformal energy space for the super-conformal case $p>5$. However it is not clear whether such scattering theory can be extended beyond the conformal invariant power $p=5$. For a more comprehensive discussion on the global dynamics for defocusing semilinear wave equation in general dimension, we refer the interested readers to  \cite{Hidano:scattering:NLW}, \cite{yang:scattering:NLW} and references therein.

Both of the above mentioned results (indeed all the previous related works) relied on the approximate conservation of conformal energy derived by using the conformal Killing vector field as multiplier. Since the conformal Killing vector field is order $2$ with weights $t^2$, an effective estimate requires that the power $p$ can not be too small. A possible way to study the asymptotic behaviors for smaller $p$ is to use vector fields with less weights. This comes in the robust new vector field method introduced by Dafermos-Rodnianski in \cite{newapp}, which has been successfully applied to study the asymptotic decay properties for the energy subcritical defocusing semilinear wave equations in higher dimensions \cite{yang:scattering:NLW}. However this method fails for equation in dimension $2$ (see detailed discussion in the end of the introduction).

With drawback of the aforementioned methods to investigate the global dynamics of equation \eqref{eq:NLW:semi:2d} in $\mathbb{R}^{1+2}$, this paper is devoted to developing new approach to study linear and nonlinear wave equations in dimension $2$. To derive weighted energy estimates with order less than $2$, we apply new vector fields to regions bounded by null hyperplanes instead of null cones used in higher dimensions. The use of null hyperplane can be viewed as reduction of the equation to $\mathbb{R}^{1+1}$, on which the wave operator is invariant under a larger class of conformal symmetry (see \cite{Yang:NLW:1d}). This enables us to use vector fields of order $\frac{p-1}{2}$ for all $1<p<5$.

Another difficulty to study the decay of linear waves in $\mathbb{R}^{1+2}$ is the failure of the embedding $L^\infty \hookrightarrow H^1$. In space dimension $3$, the work \cite{yang:NLW:ptdecay:3D} made use of the weighted potential energy flux through backward light cone, which is sufficiently strong to control the nonlinearity. As pointed out above, such estimate is absent in dimension $2$. This suggests us to explore the weighted energy estimate for the derivatives of the solution in addition to the potential energy decay.  Since a priori the $H^2$ norm of solution to \eqref{eq:NLW:semi:2d} grows at most polynomially in time $t$, we make use of Br\'{e}zis-Gallouet-Wainger inequality (logarithmic Sobolev embedding) to show that the solution decays in time $t$ for all $1<p<5$.

To state our main theorems, recall from \cite{yang:scattering:NLW} that the critical exponent
\[
s_p=\frac{p-3}{p-1}
\]
for Sobolev space. In particular, equation \eqref{eq:NLW:semi:2d} is energy subcritical for all $1<p<\infty$, and is conformal invariant when $p=5$ (corresponding to $s_p=\frac{1}{2}$). We also recall the linear operator  $\mathbf{L}(t)$
\[
\mathbf{L}(t)(\phi_0(x), \phi_1(x))=(\phi(t, x), \pa_t\phi(t, x))
\]
with $\phi$ solving the linear wave equation $\Box\phi=0$, $\phi(0, x)=\phi_0$, $\pa_t\phi(0, x)=\phi_1$ on $\mathbb{R}^{1+2}$.

For constant $0\leq \ga\leq 2$ define the weighted energy norm of the initial data
\begin{align*}
  \mathcal{E}_{k,\ga}=\sum\limits_{l\leq k}\int_{\mathbb{R}^2}(1+|x|)^{\ga+2 l}(|\nabla^{l+1}\phi_0|^2+|\nabla^l \phi_1|^2)+(1+|x|)^{\ga}|\phi_0|^{p+1}dx.
\end{align*}
In this paper we will only use the initial conformal energy $\mathcal{E}_{0, 2}$ and the first order energy $\mathcal{E}_{1, 0}$.

Our first result is the time decay of the potential energy as well as the scattering results in critical Sobolev space.
\begin{thm}
  \label{thm:main:td}
  For solution $\phi$ to the defocusing semilinear wave equation \eqref{eq:NLW:semi:2d} in $\mathbb{R}^{1+2}$ with subconformal power $1<p\leq 5$, the potential energy decays as follows
  \begin{align*}
    \int_{\mathbb{R}^2}|\phi(t, x)|^{p+1} dx\leq C\mathcal{E}_{0, 2} (1+t)^{-\frac{p-1}{2}}
    ,\quad \forall t\geq 0
  \end{align*}
  for some constant $C$ depending only on $p$.

   As a consequence for all $1+\sqrt{8}<p<5$, the solution $\phi$ scatters to linear solutions in the critical Sobolev space and energy space, that is, there exist pairs $\phi_0^{\pm}\in \dot H_x^{s_p}\cap \dot H_x^1$ and  $ \phi_1^{\pm}\in \dot H_x^{s_p-1}\cap L_x^2$ such that
\begin{align*}
      \lim\limits_{t\rightarrow\pm\infty}\|(\phi(t, x),\pa_t\phi(t,x))-\mathbf{L}(t)(\phi_0^{\pm}(x), \phi_1^{\pm}(x))\|_{\dot{H}_x^s\times \dot{H}_x^{s-1}}=0
\end{align*}
for all $s_p\leq s\leq 1$.
\end{thm}

\begin{Remark}
The superconformal case $p\geq 5$ has been well understood in \cite{Velo87:decay:NLW}, \cite{Pecher82:NLW:2d}. Since we are aiming at comparing the nonlinear solution with linear solutions, the larger power $p$ makes the problem easier due to the decay of the solution.
\end{Remark}

\begin{Remark}
The time decay of the potential energy obtained in \cite{Pecher82:NLW:2d} is
\begin{align*}
    \int_{\mathbb{R}^2}|\phi(t, x)|^{p+1} dx\leq C\mathcal{E}_{0, 2}(1+t)^{3-p}
    ,\quad \forall t\geq 0.
  \end{align*}
  For the subconformal case when $p<5$ we see that $p-3< \frac{p-1}{2}$. In particular we improved the time decay of the potential energy.
\end{Remark}

\begin{Remark}
The lower bound for the above scattering result is consistent with that in higher dimensions ($p(d)=\frac{1+\sqrt{d^2+4d-4}}{d-1}$, see the main theorem in \cite{yang:scattering:NLW}). However we still need to require that the initial data belong to the conformal energy space.
\end{Remark}

\begin{Remark}
Scattering in the critical Sobolev space $\dot{H}^{s_p}$ is motivated by the open problem that whether equation \eqref{eq:NLW:semi:2d} is globally well-posed in critical Sobolev space, see recent breakthrough \cite{Dodson:NLW:3d:p3}, \cite{Dodson:NLW:3d:p35} by Dodson in $\mathbb{R}^{1+3}$.
\end{Remark}

Based on the above time decay of the potential energy as well as Br\'{e}zis-Gallouet-Wainger inequality, we also obtain pointwise decay estimates for the solution.
\begin{thm}
\label{thm:main:pd}
 Consider the Cauchy problem to the defocusing semilinear wave equation \eqref{eq:NLW:semi:2d}. For initial data $(\phi_0, \phi_1)$ bounded in $\mathcal{E}_{0, 2}$ and $\mathcal{E}_{1, 0}$, the solution $\phi$ satisfies the following pointwise decay estimates for all $t\geq 0$:
\begin{itemize}
\item For the case when $\frac{11}{3}<p<5$, the solution enjoys the sharp time decay estimates
\begin{equation*}
|\phi(t, x)|\leq C (1+t+|x|)^{-\frac{1}{2}}
\end{equation*}
for some constant $C$ depending only on $p$, $\mathcal{E}_{0, 2}$ and $\mathcal{E}_{1, 0}$.
\item Otherwise for all $1<p\leq \frac{11}{3}$ and $\ep>0$, the solution still decays
\begin{equation*}
|\phi(t, x)|\leq C_{\ep}  (1+t+|x|)^{-\frac{p-1}{8}+\ep  }
\end{equation*}
with $C_{\ep}$ depending on  $p$, $\ep$ and $\mathcal{E}_{0, 2}$ and $\mathcal{E}_{1, 0}$.
\end{itemize}
 \end{thm}
\begin{Remark}
The constants $C$ and $C_{\ep}$ depend polynomially on $\mathcal{E}_{0, 2}$ and linearly on $\mathcal{E}_{1, 0}^{\ep}$. One can make such dependence explicitly in the course of the proof.
\end{Remark}

\begin{Remark}
In dimension $d=3$ as shown in \cite{yang:NLW:ptdecay:3D}, the solution decays as quickly as linear waves (along the cone) for $p>p(d)$. Since $\frac{11}{3}<p(2)=1+\sqrt{8}$, the sharp pointwise time decay estimate even holds for $p$ less than $p(2)$ in $\mathbb{R}^{1+2}$. In general, we conjecture that the solution to \eqref{eq:NLW:semi:2d} in $\mathbb{R}^{1+d}$ scatters in critical Sobolev space and decays sharply (only in dimension $d=2, 3$) for $1+\frac{4}{d}<p<1+\frac{4}{d-2}$.
\end{Remark}

\begin{Remark}
Pointwise decay estimates are only available in \cite{Pecher82:NLW:2d} when $p>1+\sqrt{8}$. One novelty  of our work is that the pointwise decay estimate holds for all $p>1$. Moreover, our decay estimates are stronger than those obtained by Glassey-Pecher. Hence we also improved the scattering result in energy space.
\end{Remark}

 As a corollary of the above pointwise decay estimate, we extend Glassey-Pecher's scattering result \cite{Pecher82:NLW:2d} in energy space.
\begin{cor}
 \label{cor:scattering:2D}
 For $p>2\sqrt{5}-1$ and initial data bounded in $\mathcal{E}_{0, 2} $ and $\mathcal{E}_{1, 0}$, the solution $\phi$ of the defocusing semilinear wave equation \eqref{eq:NLW:semi:2d} in $\mathbb{R}^{1+2}$ verifies
 \begin{align*}
   \|\phi\|_{L_t^p L_x^{2p}}<\infty.
 \end{align*}
 Consequently the solution scatters in energy space, that is, there exists pairs $(\phi_0^{\pm}(x), \phi_1^{\pm}(x))$ such that
 \begin{align*}
      \lim\limits_{t\rightarrow\pm\infty}\|\phi(t, x)-\mathbf{L}(t)(\phi_0^{\pm}(x), \phi_1^{\pm}(x))\|_{\dot{H}_x^1}+\| \pa_t\phi(t,x)-\pa_t \mathbf{L}(t)(\phi_0^{\pm}(x), \phi_1^{\pm}(x))\|_{ L_x^2}=0.
\end{align*}
\end{cor}

To better elaborate our idea for the proof, let's first review the existing methods to study the asymptotic decay properties for energy subcritical defocusing semilinear wave equations. The pioneering work dates back to Strauss in \cite{Strauss:NLW:decay} for the equation in $\mathbb{R}^{1+3}$ with superconformal power $3<p<5$. The key observation was the almost conservation of conformal energy obtained by using the conformal Killing vector field as multiplier. The superconformal power $1+\frac{4}{d-1}<p$ plays the role that the associated conformal energy is uniformly bounded by the initial data, which leads to the sharp time decay of the potential energy
\begin{align*}
\int_{\mathbb{R}^{d}}|\phi|^{p+1}dx\leq C (1+t)^{-2},\quad \forall d\geq 2
\end{align*}
This time decay estimate is also the starting point to conclude the pointwise decay estimates (in dimension $d=2$ or $3$) and scattering results \cite{vonWahl72:decay:NLW:super}, \cite{Velo87:decay:NLW}. To go beyond the conformal power, Pecher \cite{Pecher82:decay:3d} observed that instead of uniform bound, the conformal energy grows at most polynomially by using Gronwall's inequality. And for $p$ close to the conformal invariant power, the potential energy still decays
\begin{align*}
\int_{\mathbb{}{R}^{d}} |\phi|^{p+1}dx\leq C(1+t)^{1+d-p(d-1)},\quad \forall d \geq 2, \quad p>\frac{d+1}{d-1}.
\end{align*}
This covers part of subconformal power \cite{Pecher82:NLW:2d},
 \cite{Hidano03:scattering:NLW:56D}, \cite{Hidano:scattering:NLW}.

 Improvements both on the range of power $p$ (smaller $p$ makes the problem more difficult) and decay estimates of the solution in higher dimension $d\geq 3$  were achieved by the second author in \cite{yang:scattering:NLW}, \cite{yang:NLW:ptdecay:3D}, \cite{yang:NLW:ptdecay:3D:smallp}, which rely on the new vector field method original developed by Dafermos-Rodnianski \cite{newapp} for the study of decay estimates for linear waves on black hole background. The crucial new ingredient of this approach is a type of $r$-weighted energy estimates obtained by using the vector fields $r^{\ga}(\pa_t+\pa_r)$ as multipliers for all $0\leq \gamma\leq 2$.
This leads to the time decay of the potential energy (comparatively, the original estimate is an integral version in time $t$)
\begin{align*}
\int_{\mathbb{R}^{d}}|\phi(t, x)|^{p+1}dx\leq C (1+t)^{-\frac{(d-1)(p-1)}{2}}.
\end{align*}
One notices that for the subconformal case when $p<\frac{d+3}{d-1}$
\begin{align*}
\frac{(d-1)(p-1)}{2}>1+d-p(d-1).
\end{align*}
In particular the new approach greatly improved the time decay of the potential energy for the subconformal case.

However this new approach only works for the equation in higher dimensions $\mathbb{R}^{1+d}$, $d\geq 3$ and fails when $d=2$ for the reason that the lower order term generated by the use of vector field $r^{\ga}(\pa_t+\pa_r)$ with $1<\gamma <2$ switches sign to be negative.

To overcome this difficulty, we first apply the vector field
\begin{align*}
X= (x_2^2+(t-x_1)^2+1)\partial_t+(x_2^2-(t-x_1)^2)\partial_1+2(t-x_1)x_2\partial_2
\end{align*}
to the region bounded by the null hyperplane $\{t=x_1\}$ and the initial hypersurface, which leads to the weighted energy estimate
\begin{equation*}
  \int_{t\geq 0}x_2^2|(\pa_t+\pa_1)\phi|^2(t,t,x_2) +|\pa_2\phi(t, t, x_2)|^2+\frac{2}{p+1}|\phi(t, t, x_2)|^{p+1} dtdx_2
\leq  C\mathcal{E}_{0, 2}.
\end{equation*}
Then we consider the vector field
 \begin{align*}
X=u_1^{\frac{p-1}{2}}(\partial_t-\partial_1)+u_1^{\frac{p-1}{2}-2}x_2^2(\partial_t+\partial_1)+2u_1^{\frac{p-1}{2}-1}x_2\partial_2\end{align*}
with $u_1=t-x_1+1$ and region bounded by the null hyperplane $\{t=x_1\}$, the constant time $t$ hypersurface $\{t=t\}$ as well as the initial hypersurface. This implies the weighted energy estimate
\begin{align*}
&\int_{x_1\leq t} u_1^{\frac{p-5}{2}}(x_2^2+u_1^2)|\phi(t, x)|^{p+1} dx\\
&\leq C\mathcal{E}_{0, 2}+C\int_{t\geq 0}(|x_2(\pa_t+\pa_1)\phi|^2+|\partial_2\phi|^2+\frac{2}{p+1}|\phi|^{p+1})(t,t,x_2)dtdx_2\\
&\leq C\mathcal{E}_{0, 2}
\end{align*}
in view of the previous weighted energy estimate through the null hyperplane $\{t=x_1\}$. Since we only study the equation in the future $t\geq 0$, the above estimate in particular shows that
\begin{align*}
\int_{x_1\leq 0} (1+t)^{\frac{p-1}{2}} |\phi(t, x)|^{p+1} dx\leq  C\mathcal{E}_{0, 2}
\end{align*}
as $u_1=t-x_1+1$. By symmetry ($x_1$ to $-x_1$) the above estimate also holds on the region $x_1\geq 0$ (the time $t$ is always fixed). This demonstrates the improved time decay of the potential energy at time $t$.

\bigskip
For the pointwise decay estimate for the solution, we go back to the conformal energy identity to derive that
\begin{align*}
 (1+t)^2\|\nabb \phi\|_{L^2}^2+\|(1+|t-r|)\nabla \phi\|_{L^2}^2+\|\phi\|_{L^2}^2 \leq C
 (1+t)^{\frac{5-p}{2}}
 \end{align*}
 in view of the above time decay of the potential energy. Here $\nabb=\nabla-r^{-1}x\pa_r$ is the angular derivative. The rough pointwise decay estimates for the solution then follow by use of the following Br\'{e}zis-Gallouet-Wainger inequality
 \begin{align*}
      \|u\|_{L^\infty(\mathbb{R}^2)}\leq C \| u\|_{H^1(\mathbb{R}^2)}\left(1+\ln \frac{\|u\|_{H^2(\mathbb{R}^2)}}{\| u\|_{H^1(\mathbb{R}^2)}}\right)^{\f12}
    \end{align*}
adapted to regions away from the light cone $\{t=|x|\}$ and annulus near the light cone, see details in Lemma \ref{lem:log:Sobolev}. This is sufficiently strong to conclude the pointwise decay estimate
\[
|\phi(t, x)|\leq C (1+t)^{-\frac{p-1}{8}+\ep}
\]
for all $1<p<5$. The loss of decay rate $(1+t)^{\ep}$ arises due to the at most polynomial growth of the second order energy $\|\phi(t, x)\|_{H^2}$.

Finally for the improved sharp decay estimate for larger $p$, inspired by the idea from \cite{Moncrief1}, we  write the nonlinear solution as sum of linear solution and contribution from nonlinearity.  By carefully estimating the nonlinearity, a type of Gronwall's inequality then leads to sharp pointwise decay of the solution when $p>\frac{11}{3}$.

The plan of the paper is as follows: In Section 2, we define some notations and recall the energy method for wave equations. In Section 3, we use vector field method applied to hyperplane instead of cones to derive time decay estimates for the potential energy. In section 4, we develop logarithmic Sobolev embedding based on Br\'{e}zis-Gallouet-Wainger inequaltiy and conclude the pointwise decay estimates for the solution.

\textbf{Acknowledgments.} The authors would like to thank R. Frank for kindly pointing out the origin of Br\'{e}zis-Gallouet-Wainger inequality. S. Yang is partially supported by NSFC-11701017.

\section{Preliminaries and energy identities}
\label{notation}
Let's first define some necessary notations.
We use the standard Cartesian coordinates $(t, x_1, x_2)$ and polar local coordinate system $(t, r, \theta)$ of Minkowski space.
The initial hypersurface $\{t=0\}$ will be denoted as $\Sigma_0$ and $\Sigma_t$ stands for the constant time $t$ hypersurface.
Since the wave equation is time reversible, without loss of generality we only prove estimates in the future, i.e., $t\geq 0$.

As discussed in the introduction, the main idea to study such large data problem is to start with some weighted energy estimate, which could be obtained by using vector field method in the following way:
Define the associated energy momentum tensor for the scalar field $\phi$
\begin{align*}
  T[\phi]_{\mu\nu}=\pa_{\mu}\phi\pa_{\nu}\phi-\f12 m_{\mu\nu}(\pa^\ga \phi \pa_\ga\phi+\frac{2}{p+1} |\phi|^{p+1}),
\end{align*}
where $m_{\mu\nu}$ is the flat Minkowski metric on $\mathbb{R}^{1+2}$.

For any vector field $X$ and any function $\chi$, define the current
\begin{equation*}
J^{X,  \chi}_\mu[\phi]=T[\phi]_{\mu\nu}X^\nu -
\f12\pa_{\mu}\chi \cdot|\phi|^2 + \f12 \chi\pa_{\mu}|\phi|^2.
\end{equation*}
By using Stokes' formula, we derive the energy identity
\begin{equation}
\label{eq:energy:id}
\int_{\pa\mathcal{D}}i_{  J^{X, \chi}[\phi]} d\vol =\iint_{\mathcal{D}}  T[\phi]^{\mu\nu}\pi^X_{\mu\nu}+
\chi \pa_\mu\phi\pa^\mu\phi -\f12\Box\chi\cdot|\phi|^2 +\chi \phi\Box\phi+X(\phi)(\Box\phi-|\phi|^{p-1}\phi) d\vol
\end{equation}
for any domain $\mathcal{D}$ in $\mathbb{R}^{1+2}$. Here $\pi^X=\f12 \cL_X m$  is the deformation tensor of the metric $m$ along the vector field $X$ and $i_Z d\vol$ is the contraction of the vector field with the volume form $d\vol$. By suitably choosing vector field $X$ and function $\chi$, we derive weighted energy estimate for solution of \eqref{eq:NLW:semi:2d}. The ultimate goal is to show the time decay of the potential energy.

Once we have time decay estimate for the potential energy, we rely on the following representation formula for linear wave equation  with initial data $(\phi_0, \phi_1)$
\begin{equation}
\label{eq:rep:wave:2D}
\begin{split}
  2\pi\phi(t, x)=&t\int_{|y|<1}\frac{\phi_1(x+ty)}{\sqrt{1-|y|^2}}dy+\pa_t\left(t\int_{|y|<1}\frac{\phi_0(x+ty)}{\sqrt{1-|y|^2}}dy\right)\\
  &+\int_0^{t}(t-s)\int_{|y|<1}\frac{\Box\phi(s, x+(t-s)y)}{\sqrt{1-|y|^2}}dy ds.
  \end{split}
\end{equation}
Finally we make a convention through out this paper to avoid too many constants that $A\les B$ means that there exists a constant $C$, depending possibly on $p$, $\mathcal{E}_{0, 2}$, $\mathcal{E}_{1, 0}$ such that $A\leq CB$.

\section{Improved time decay of the potential energy}
For linear wave equations in space dimension greater than two, one can obtain the integrated local energy decay estimates as well as energy flux decay estimates through Dafermos-Rodnianski's modified vector field method \cite{newapp}. However such method fails in space dimension $2$ due to the wrong sign of lower order terms. So far as the authors know, it is even not clear whether the associated version holds for linear waves in $\mathbb{R}^{1+2}$, at least through vector field method.
Since our goal is to investigate the asymptotic decay properties for solutions of defocusing semilinear wave equation, we instead derive new estimates by introducing new vector fields to show the time decay of the potential energy, which substantially improves the existing time decay obtained by using the conformal Killing vector field as multiplier. As the superconformal case $p\geq 5$ is easier and has been well understood in \cite{Pecher82:NLW:2d}, in the sequel we only study the solution with subconformal power
\[
1<p<5. 
\]
We begin in this section with the improved time decay of the potential energy.
\begin{Prop}
  \label{prop:timedecay}
  For solution $\phi$ to the defocusing semilinear wave equation \eqref{eq:NLW:semi:2d} in $\mathbb{R}^{1+2}$ with subconformal power $1<p<5$, the potential energy decays as follows
  \begin{align*}
    \int_{\mathbb{R}^2}|\phi(t, x)|^{p+1}(1+t+|x|)^{\frac{p-1}{2}}dx\leq C\mathcal{E}_{0, 2}
    ,\quad \forall t\geq 0
  \end{align*}
  for some constant $C$ depending only on $p$.

   As a consequence for all $1+\sqrt{8}<p<5$, the solution $\phi$ is bounded in the following spacetime norm
  \begin{align}
  \label{eq:spacetimebd:phi}
    \|\phi\|_{L^{\frac{3(p-1)}{2}}(\mathbb{R}^{1+2})}\leq C \mathcal{E}_{0, 2}^{\frac{1}{p+1}}
  \end{align}
  for some constant $C$ relying only on $p$.
\end{Prop}

\begin{proof}
In the energy identity \eqref{eq:energy:id}, choose the vector field $X$ and the function $\chi$ as follows:
\begin{align*}
X=(r^2-t^2)(\partial_t+\partial_1)+2(t-x_1)S&=(x_2^2+(t-x_1)^2)\partial_t+(x_2^2-(t-x_1)^2)\partial_1+2(t-x_1)x_2\partial_2,\\
& \chi=t-x_1.
\end{align*}
Here $r=|x|$ and $S=t\pa_t+r\pa_r$ is the spacetime scaling vector field.
We then compute that
\[
\pi^X= 2(t-x_1)m, \quad \partial_t\chi=1, \quad (\partial_t+\partial_1)\chi=0,\quad \Box\chi=0.
\]
where $m$ is the Minkowski metric on $\mathbb{R}^{1+2}$. Thus the bulk term on the right hand side of the energy identity \eqref{eq:energy:id} is
\begin{align*}
& T[\phi]^{\mu\nu}\pi^X_{\mu\nu}+
\chi \pa_\mu\phi \pa^\mu\phi  +\chi\phi\Box\phi-\f12 \Box\chi |\phi|^2\\
&=-\frac{6(t-x_1)}{p+1}|\phi|^{p+1} +(t-x_1) |\phi|^{p+1}= -\frac{(5-p)(t-x_1)}{p+1}|\phi|^{p+1}.
\end{align*}
In particular, the bulk integral is nonnegative when $p\leq 5$ and $ t\leq x_1$.

Next we need to compute the left hand side of the energy identity \eqref{eq:energy:id} for the region $\mathcal{D}=\{0\leq t\leq \min(x_1,T)\}$ with boundary $\pa\mathcal{D}$ consisting of the null hyperplane $\{t-x_1=0\}$ and the constant $t$-slices.
 Recall the volume form
\[
d\vol=dtdx=dtdx_1dx_2 .
\]
\def\L{L}
\def\Lu{\underline{L}}
 Hence for the null hypersurface $\{t-x_1=0\}$, we can compute that
\begin{align*}
i_{J^{X, \chi}[\phi]}d\vol&=-(J^{X, \chi}[\phi])_{L_1}dtdx_2= -( T[\phi]_{\L_1 \nu}X^\nu -
\f12 \L_1\chi |\phi|^2 + \f12 \chi\cdot \L_1 |\phi|^2)  dtdx_2\\
&=-T[\phi]_{\L_1 X} dtdx_2=-x_2^2T[\phi]_{\L_1\L_1} dtdx_2=-x_2^2|\L_1\phi|^2dtdx_2.
\end{align*}
Here we denote $ \L_1=\partial_t+\partial_1$ and have used the fact that on the hyperplane $\{t-x_1=0\}$
\[
\L_1\chi=0,\quad \chi=0,\quad X=x_2^2\L_1.
\]
Then on the constant $t$-slice, we can show that
\begin{align*}
  i_{J^{X,\chi}[\phi]}d\vol=&(J^{X, \chi}[\phi])^{0}dx= -( T[\phi]_{0 \nu}X^\nu -
\f12 \pa_t\chi |\phi|^2 + \f12 \chi\cdot \pa_t |\phi|^2)  dx\\
=&-\f12\big( (x_2^2+(t-x_1)^2) (|\partial_t\phi|^2+|\nabla\phi|^2+\frac{2}{p+1}|\phi|^{p+1})- |\phi|^2 +(t-x_1)   \pa_t |\phi|^2  \\
&\quad \qquad+2\partial_t  \phi((x_2^2-(t-x_1)^2)\partial_1\phi+2(t-x_1)x_2\partial_2\phi)\big)  dx\\
=&-\f12( |x_2(\partial_t\phi+\partial_1\phi)+(t-x_1)\partial_2\phi|^2+|(t-x_1)(\partial_t\phi-\partial_1\phi)+x_2\partial_2\phi+\phi|^2\\
&\qquad +\frac{2(x_2^2+(t-x_1)^2)}{p+1}|\phi|^{p+1}  -\partial_1((x_1-t)\phi^2)-\partial_2(x_2\phi^2))  dx.
\end{align*}
Observe that on the constant $t$-slice, we also have
\begin{align*}
  \int_{x_1\geq t} (\partial_1((x_1-t)\phi^2)+\partial_2(x_2\phi^2))(t,x)dx=0.
\end{align*}
This holds whenever $\phi$ decays suitably at spatial infinity. Alternatively, one can choose bounded region and then take limit.
The above computations together with the energy identity \eqref{eq:energy:id} then lead to the following weighted energy identity for $\mathcal{D}=\{0\leq t\leq \min(x_1,T)\}$
\begin{align*}
  &\iint_{0\leq t\leq \min(x_1,T)}  \frac{(5-p)(x_1-t)}{p+1}|\phi|^{p+1}+\int_{0\leq t\leq T}x_2^2|\L_1\phi|^2(t,t,x_2)dtdx_2 \\
  &+\f12\int_{x_1\geq t}  (|x_2 L_1\phi+(t-x_1)\partial_2\phi|^2+|(t-x_1)(\pa_t\phi-\pa_1\phi)+x_2\partial_2\phi+\phi|^2
  \\&\qquad
  +\frac{2(x_2^2+(t-x_1)^2)}{p+1}|\phi|^{p+1} ) dx\Big|_{t=0}^{t=T}=0.
\end{align*}
Letting $T\to+\infty$ we conclude that
\begin{equation}
\label{eq:E1}
\begin{split}
  &\int_{t\geq 0}x_2^2|\L_1\phi|^2(t,t,x_2)dtdx_2\\
  &\leq \f12\int_{x_1\geq 0}  (|x_2(\phi_1+\partial_1\phi_0)-x_1\partial_2\phi_0|^2  +|-x_1(\phi_1-\partial_1\phi_0)+x_2\partial_2\phi_0+\phi_0|^2\\
  &\qquad \qquad
+\frac{2(x_2^2+x_1^2)}{p+1}|\phi_0|^{p+1} ) dx\\
&\leq C  \int_{\{t=0\}}(1+|x|^2)(|\phi_1|^2+|\nabla\phi_0|^2+\frac{2}{p+1}|\phi_0|^{p+1})dx\\
&= C\mathcal{E}_{0, 2}.
\end{split}
\end{equation}
Here and in the following the constant $C$ relies only on $p$ and some small positive constant $\ep$ (in case such constant appears).
In last step the integral of $|\phi_0|^2$ can be controlled by the weighted energy by using Hardy's inequality.

\bigskip

Next in the energy identity \eqref{eq:energy:id}, choose the vector field $X=\partial_t,\ \chi=0,$ then $\pi^X=0,$ and
\begin{align*}
& T[\phi]^{\mu\nu}\pi^X_{\mu\nu}+
\chi \pa_\mu\phi \pa^\mu\phi  +\chi\phi\Box\phi-\f12 \Box\chi |\phi|^2=0.
\end{align*}
For the null hypersurface $\{t-x_1=0\}$, we have
\begin{align*}
i_{J^{X, \chi}[\phi]}d\vol&=-(J^{X, \chi}[\phi])_{L_1}dtdx_2=-T[\phi]_{\L_1 X} dtdx_2\\
&=-(\pa_{t}\phi(\pa_{t}+\pa_{1})\phi+\f12(-|\partial_t\phi|^2+|\nabla\phi|^2+\frac{2}{p+1}|\phi|^{p+1}) )dtdx_2\\
&=-\f12(|\L_1\phi|^2+|\partial_2\phi|^2+\frac{2}{p+1}|\phi|^{p+1}) dtdx_2.
\end{align*}
On the constant $t$-slice, we can show that
\begin{align*}
  i_{J^{X,\chi}[\phi]}d\vol=&(J^{X, \chi}[\phi])^{0}dx= -( T[\phi]_{0 \nu}X^\nu)  dx\\
=&-\f12(|\partial_t\phi|^2+|\nabla\phi|^2+\frac{2}{p+1}|\phi|^{p+1})   dx.
\end{align*}
We therefore derive that
\begin{align*}
  &\int_{0\leq t\leq T}(|\L_1\phi|^2+|\partial_2\phi|^2+\frac{2}{p+1}|\phi|^{p+1})(t,t,x_2)dtdx_2 \\
  &+\int_{x_1\geq t}  (|\partial_t\phi|^2+|\nabla\phi|^2+\frac{2}{p+1}|\phi|^{p+1}) dx\Big|_{t=0}^{t=T}=0.
\end{align*}
Similarly by letting $T\to+\infty$, we conclude that
\begin{align}
\label{eq:E2}
  &\int_{t\geq 0}(|\L_1\phi|^2+|\partial_2\phi|^2+\frac{2}{p+1}|\phi|^{p+1})(t,t,x_2)dtdx_2\\
  \nonumber &\leq \int_{x_1\geq 0}  (|\phi_1|^2+|\nabla\phi_0|^2+\frac{2}{p+1}|\phi_0|^{p+1}) dx\leq \mathcal{E}_{0, 0}.
\end{align}
Finally in the energy identity \eqref{eq:energy:id}, choose the vector field $X$ and the function $\chi$ as follows:
\begin{align*}
X&=u_1^q(\partial_t-\partial_1)+u_1^{q-2}x_2^2(\partial_t+\partial_1)+2u_1^{q-1}x_2\partial_2, \\
 \chi&=u_1^{q-1},\quad u_1=t-x_1+1,\quad q=\f12(p-1).
\end{align*}
 Here we assume $t\geq x_1$. Recall that $ \L_1=(\partial_t+\partial_1),\ \Lu_1=(\partial_t-\partial_1)$. We then compute that
\begin{align*}
\nabla_{\L_1}X= 0,\quad &\nabla_{\Lu_1}X= 2qu_1^{q-1}\Lu_1+2(q-2)u_1^{q-3}x_2^2\L_1+4(q-1)u_1^{q-2}x_2\partial_2,\\
&\nabla_{2}X=2u_1^{q-2}x_2\L_1+2u_1^{q-1}\partial_2.
\end{align*}
In particular the non-vanishing components of the deformation tensor $\pi_{\mu\nu}^X$ are
\begin{align*}
\pi^X_{L_1\Lb_1}&=-2qu_1^{q-1},\quad \pi^X_{\Lb_1\Lb_1}=-4(q-2)u_1^{q-3}x_2^2,\\
 \pi^X_{\Lb_1 \pa_2}&=2(q-2)u_1^{q-2}x_2,\quad \pi^X_{\pa_2\pa_2}=2u_1^{q-1}.
\end{align*}
Since $\Box\chi=0 $, we therefore can compute that
\begin{align*}
& T[\phi]^{\mu\nu}\pi^X_{\mu\nu}+
\chi \pa_\mu\phi \pa^\mu\phi  +\chi\phi\Box\phi-\f12 \Box\chi |\phi|^2\\
&=-qu_+^{q-1} (|\partial_2(\phi)|^2+\frac{2}{p+1}|\phi|^{p+1})+ 2u_1^{q-1}(|\partial_2(\phi)|^2-\frac{1}{2}\pa^\mu \phi \pa_\mu\phi-\frac{1}{p+1}|\phi|^{p+1})\\
&\quad -(q-2)u_1^{q-3}x_2^2 |\L_1\phi|^2-2(q-2)u_1^{q-2}x_2\partial_2\phi\cdot \L_1\phi+
u_1^{q-1} \pa_\mu\phi \pa^\mu\phi  +u_1^{q-1} |\phi|^{p+1} \\
&=(2-q)u_1^{q-3} |x_2\L_1\phi+u_1\partial_2\phi|^2 .
\end{align*}
Here we used $2q+2=p+1.$ In particular when $p\leq 5$, we have $q\leq 2$ and the bulk integral is nonnegative.

Now we compute the boundary integrals for the region $\mathcal{D}=\{\max(x_1,0)\leq t\leq T\}$. On the null hyperplane $\{t-x_1=0\}$, we have
\[
\chi= u_1=1, \quad X=\Lu_1+x_2^2\L_1+2x_2\partial_2.
\]
Therefore we have
\begin{align*}
i_{J^{X, \chi}[\phi]}d\vol&=-(J^{X, \chi}[\phi])_{L_1}dtdx_2= -( T[\phi]_{\L_1 \nu}X^\nu -
\f12 \L_1\chi |\phi|^2 + \f12 \chi\cdot \L_1 |\phi|^2)  dtdx_2\\
&=-(x_2^2T[\phi]_{\L_1\L_1}+T[\phi]_{\L_1\Lu_1}+2x_2T[\phi]_{\L_1 \pa_2}+ \f12 \L_1 |\phi|^2) dtdx_2\\
&=-(x_2^2|\L_1\phi|^2+|\partial_2(\phi)|^2+\frac{2}{p+1}|\phi|^{p+1}+2x_2\partial_2\phi\cdot \L_1\phi+ \f12 \L_1 |\phi|^2) dtdx_2\\
&=-(|x_2\L_1\phi+\partial_2\phi|^2+\frac{2}{p+1}|\phi|^{p+1}+ \f12 \L_1 |\phi|^2) dtdx_2.
\end{align*}
On the constant $t$-slice, we can show that
\begin{align*}
  i_{J^{X,\chi}[\phi]}d\vol=&(J^{X, \chi}[\phi])^{0}dx= -( T[\phi]_{0 \nu}X^\nu -
\f12 \pa_t\chi |\phi|^2 + \f12 \chi\cdot \pa_t |\phi|^2)  dx\\
=&-\f12u_1^{q-2}\big( (x_2^2+u_1^2) (|\partial_t\phi|^2+|\nabla\phi|^2+\frac{2}{p+1}|\phi|^{p+1})- (q-1)|\phi|^2  +u_1   \pa_t |\phi|^2 \\
&\qquad\qquad  +2 (x_2^2-u_1^2)\pa_t\phi\cdot \partial_1\phi+4u_1 x_2 \partial_t\phi \cdot\partial_2\phi \big)  dx\\
=&-\f12 u_1^{q-2}\big( |x_2(\partial_t\phi+\partial_1\phi)+u_1\partial_2\phi|^2+|u_1(\partial_t\phi-\partial_1\phi)+x_2\partial_2\phi+\phi|^2\\
&\qquad \qquad+\frac{2(x_2^2+u_1^2)}{p+1}|\phi|^{p+1} \big)dx- \f12(\partial_1(u_1^{q-1}\phi^2)-\partial_2(x_2u_1^{q-2}\phi^2))  dx.
\end{align*}
On the boundary $\pa\cD$ (for $\cD=\{\max(x_1,0)\leq t\leq T\}$) with suitable decay of the solution $\phi$ at spatial infinity, note that
\begin{align*}
  \int_{x_1\leq t}(\partial_1(u_1^{q-1}\phi^2)-\partial_2(x_2u_1^{q-2}\phi^2))dx\Big|_{t=0}^{t=T}=\int_{0\leq t\leq T} (\L_1 |\phi|^2)(t,t,x_2)dtdx_2.
\end{align*}
Denote
\begin{align*}
  E[\phi]_{q}(t)=\int_{x_1\leq t}u_1^{q-2}( & |x_2(\partial_t\phi+\partial_1\phi)+u_1\partial_2\phi|^2+|u_1(\partial_t\phi-\partial_1\phi)+x_2\partial_2\phi+\phi|^2\\
  &+\frac{2(x_2^2+u_1^2)}{p+1}|\phi|^{p+1})(t,x)dx.
\end{align*}
Then the above computations together with the energy identity \eqref{eq:energy:id} lead to the following weighted energy identity
\begin{align*}
  &\iint_{\max(x_1,0)\leq t\leq T}  (2-q)u_1^{q-3} |x_2\L\phi+u_1\partial_2\phi|^2+\f12(E[\phi]_{q}(T)-E[\phi]_{q}(0))\\
  &-\int_{0\leq t\leq T}(|x_2\L_1\phi+\partial_2\phi|^2+\frac{2}{p+1}|\phi|^{p+1})(t,t,x_2)dtdx_2=0,
\end{align*}
from which we conclude that
\begin{align*}
  &E[\phi]_{q}(T)\leq E[\phi]_{q}(0)+2\int_{0\leq t\leq T}(|x_2\L_1\phi+\partial_2\phi|^2+\frac{2}{p+1}|\phi|^{p+1})(t,t,x_2)dtdx_2.
\end{align*}
Note that on the initial hypersurface $\{t=0\}$,
\begin{align*}
  E[\phi]_{q}(0)&=\int_{x_1\leq 0}(1-x_1)^{q-2}( |x_2(\phi_1+\partial_1\phi_0)+(1-x_1)\partial_2\phi_0|^2+\frac{2(x_2^2+(1-x_1)^2)}{p+1}|\phi_0|^{p+1}\\
  &\qquad \qquad +|(1-x_1)(\phi_1-\partial_1\phi_0)+x_2\partial_2\phi_0+\phi_0|^2
  )dx\\
  &\leq C \mathcal{E}_{0, 2}.
\end{align*}
Now by using estimates \eqref{eq:E1} and \eqref{eq:E2},
we can bound the integral on the null hyperplane
\begin{align*}
  &2\int_{0\leq t\leq T}(|x_2\L_1\phi+\partial_2\phi|^2+\frac{2}{p+1}|\phi|^{p+1})(t,t,x_2)dtdx_2\\
  \leq&4\int_{0\leq t\leq T}(x_2^2|\L_1\phi|^2+|\partial_2\phi|^2+\frac{2}{p+1}|\phi|^{p+1})(t,t,x_2)dtdx_2\\
  \leq &  C\mathcal{E}_{0, 2}.
\end{align*}
We therefore conclude that
\begin{align*}
  \int_{x_1\leq T}\frac{2(u_1^{q}+u_1^{q-2}x_2^2)|\phi|^{p+1}}{p+1}dx\leq E[\phi]_{q}(T)\leq C \mathcal{E}_{0, 2},\quad \forall T\geq 0.
\end{align*}
Recall that $u_1=t-x_1+1$. By restricting the integral to the region $\{x_1\leq 0\}$, the above estimate in particular implies that
\begin{align*}
  \int_{x_1\leq 0}\frac{2(t+1)^q}{p+1}|\phi|^{p+1}(t,x)dx &\leq \int_{x_1\leq t}\frac{2u_1^q}{p+1}|\phi|^{p+1}(t,x)dx \leq C \mathcal{E}_{0, 2},\\
  \int_{x_1\leq t\leq r}\frac{2 (1+2r)^{q-2}x_2^2|\phi|^{p+1}}{p+1}dx&\leq \int_{x_1\leq t}\frac{2 u_1^{q-2}x_2^2|\phi|^{p+1}}{p+1}dx \leq C E[\phi]_{q}(T)\leq C \mathcal{E}_{0, 2}.
\end{align*}
Here $r=|x|=\sqrt{x_1^2+x_2^2}$ and $0<q=\frac{p-1}{2}<2$.
This in particular shows that
\begin{align*}
  \int_{x_1\leq 0} (1+t)^{\frac{p-1}{2}} |\phi|^{p+1}(t,x)dx+\int_{x_1\leq 0, r\geq t} (1+r)^{\frac{p-1}{2}-2} x_2^2 |\phi|^{p+1}(t,x)dx\leq \mathcal{E}_{0, 2}.
\end{align*}
By symmetry (or changing variable $x_1\rightarrow -x_1$), we also have
\begin{align*}
   \int_{x_1\geq 0} (1+t)^{\frac{p-1}{2}} |\phi|^{p+1}(t,x)dx+\int_{x_1\geq 0, r\geq t} (1+r)^{\frac{p-1}{2}-2} x_2^2 |\phi|^{p+1}(t,x)dx\leq C \mathcal{E}_{0, 2}.
\end{align*}
Therefore we derive that
\begin{align*}
   \int_{\mathbb{R}^2} (1+t)^{\frac{p-1}{2}} |\phi|^{p+1}(t,x)dx+\int_{r\geq t} (1+r)^{\frac{p-1}{2}-2} x_2^2 |\phi|^{p+1}(t,x)dx\leq C \mathcal{E}_{0, 2}.
\end{align*}
By symmetry again (switching the variable $x_1$ and $x_2$), we also have
\begin{align*}
   \int_{r\geq t} (1+r)^{\frac{p-1}{2}-2} x_1^2 |\phi|^{p+1}(t,x)dx\leq C \mathcal{E}_{0, 2}.
\end{align*}
Adding the above two estimate leads to the time decay of the potential energy
\begin{align*}
  \int_{\R^2}(1+t+r)^{\frac{p-1}{2}}|\phi|^{p+1}(t,x)dx\leq C \mathcal{E}_{0, 2}.
\end{align*}
By our convention, the constant relies only on $p$. We hence have shown the time decay of the potential energy of the Proposition.

\bigskip

Now based on the time decay estimate of the solution, we demonstrate that the solution is bounded in the spacetime norm $L^{\frac{3(p-1)}{2}}(\mathbb{R}^{1+2})$ for all $1+\sqrt{8}<p<5$. Without loss of generality, it suffices to prove the estimate in the future $t\geq 0$. By using H\"older's inequality, we can show that
\begin{align*}
  \|\phi\|_{L^{\frac{3(p-1)}{2}}(\mathbb{R}^{1+2}\cap \{t\geq 0\})} &\leq \| \|(1+t+r)^{\frac{p-1}{2(p+1)}}\phi\|_{L_x^{p+1}}\|(1+t+r)^{-\frac{p-1}{2(p+1)}}\|_{L_x^{q_1}}\|_{L_t^{\frac{3(p-1)}{2}}(\{t\geq 0\})}\\
  &\leq C \mathcal{E}_{0, 2}^{\frac{1}{p+1}}\| \|(1+t+r)^{-\frac{p-1}{2(p+1)}}\|_{L_x^{q_1}}\|_{L_t^{\frac{3(p-1)}{2}}(\{t\geq 0\})},
\end{align*}
where
\begin{align*}
  \frac{1}{q_1}=\frac{2}{3(p-1)}-\frac{1}{p+1}=\frac{5-p}{3(p-1)(p+1)}.
\end{align*}
Since $p>1+\sqrt{8}$, we in particular have that
\begin{align*}
  2<\frac{3(p-1)^2}{2(5-p)}.
\end{align*}
 We therefore can compute that
\begin{align*}
  \|(1+t+r)^{-\frac{p-1}{2(p+1)}}\|_{L_x^{q_1}}& =\left(\int_{\mathbb{R}^2} (1+t+r)^{-\frac{p-1}{2(p+1)}\frac{3(p-1)(p+1)}{5-p}}dx\right)^{\frac{1}{q_1}}\\
  &\leq C (1+t)^{\frac{17+2p-3p^2}{6(p-1)(p+1)}}.
\end{align*}
This leads to
\begin{align*}
  \|\phi\|_{L^{\frac{3(p-1)}{2}}(\mathbb{R}^{1+2}\cap \{t\geq 0\})} & \leq C \mathcal{E}_{0, 2}[\phi]^{\frac{1}{p+1}} \left(\int_0^{\infty} (1+t)^{\frac{17+2p-3p^2}{6(p-1)(p+1)} \cdot\frac{3(p-1)}{2}} dt\right)^{\frac{2}{3(p-1)}}\\
  &\leq C \mathcal{E}_{0, 2}^{\frac{1}{p+1}}
\end{align*}
in view of the relation
\begin{align*}
17+2p-3p^2<-4(p+1)
\end{align*}
by the assumption $p>1+\sqrt{8}$. We hence
finished the proof for Proposition \ref{prop:timedecay}.
\end{proof}
The time decay estimate of main Theorem \ref{thm:main:td} follows from Proposition \ref{prop:ptdecay:allp} immediately. The scattering results of Theorem \ref{thm:main:td} is a consequence of the spacetime bound \eqref{eq:spacetimebd:phi}. For details regarding this, we refer to Proposition 5.1 of \cite{yang:scattering:NLW}. We thus have shown Theorem \ref{thm:main:td}.

\section{Pointwise decay estimate of the solution}
In this section, we rely on the potential energy decay obtained in the previous section to show the pointwise decay estimate for the solution. As we have discussed in the introduction, the large power $p$ makes the problem easier due to the better decay of the nonlinearity. And it is expected that for sufficiently large $p$, the solution behaves like linear wave as shown in \cite{Pecher82:NLW:2d}, \cite{Velo87:decay:NLW} for the superconformal case $p\geq 5$. For the subconformal case studied in this paper, the solution decays faster for larger $p$. The treatment varies for different range of $p$.
We first demonstrate a rough decay estimate for the solution for all $1<p<5$.

\subsection{Rough decay estimate for all $p$}
The obvious challenge in $\mathbb{R}^2$ is the failure of traditional Sobolev inequality for $L^\infty$ estimate by the energy. This can be compensated by Br\'{e}zis-Gallouet-Wainger inequality.

In $\mathbb{R}^2$, for $0\leq s<R$, denote
\begin{align*}
  B_R=\{x:|x|\leq R\},\quad A_{R, s}=\{x: R-s \leq |x|\leq R+s\}.
\end{align*}
\begin{Lem}
  [Br\'{e}zis-Gallouet-Wainger inequality]
  \label{lem:log:Sobolev}
  In $\mathbb{R}^{2}$, we have the following logarithmic Sobolev inequalities
\begin{align*}
      &\|u\|_{L^\infty(\mathbb{R}^2)}\leq C \| u\|_{H^1(\mathbb{R}^2)}\left(1+\ln \frac{\|u\|_{H^2(\mathbb{R}^2)}}{\| u\|_{H^1(\mathbb{R}^2)}}\right)^{\f12},\\&\|u\|_{L^\infty(B_{1/2})}\leq C \| u\|_{H^1(B_{3/4})}\left(1+\ln \frac{\|u\|_{H^2(B_{3/4})}}{\| u\|_{H^1(B_{3/4})}}\right)^{\f12},\\&\|u\|_{L^\infty(B_1^c)}\leq C \| u\|_{H^1(B_{5/6}^c)}\left(1+\ln \frac{\|u\|_{H^2(B_{5/6}^c)}}{\| u\|_{H^1(B_{5/6}^c)}}\right)^{\f12},\\&\|u\|_{L^{\infty}(A_{1,1/2})}^2\leq C\|u\|_{H^1(\R^2)}\|(u,\nabb u)\|_{L^2(\R^2)}\left(1+\ln\frac{\|u\|_{H^{2}(\R^2)}} {\|(u,\nabb u)\|_{L^2(\R^2)}}\right).
    \end{align*}
\end{Lem}
The first inequality originally appeared in
\cite{Brezis80:LSobolev} with generalizations in \cite{Brezis80:LSobolev:g}. For our needs, we will use the second inequality for the region near the axis $|x|=0$ (for fixed time $t$), while the third inequality will be applied on the region far away from the light cone. Near the light cone, the better decay of the angular derivative plays a crucial role. And we rely on the fourth inequality to prove the decay estimate for the solution. For readers' interests, we will reprove the first inequality, based on which we show the rest ones needed in this paper.

\begin{proof}
We define the Fourier transform as
\begin{align*}
      \hat{u}(\xi)=\frac{1}{2\pi}\int_{\R^2}u(x)e^{-ix\cdot\xi}dx,\quad \xi\in\R^2.
    \end{align*}
In particular we have reverse transform
\begin{align*}
u(x)=\frac{1}{2\pi}\int_{\R^2}\hat{u}(\xi)e^{ix\cdot\xi}d\xi.
\end{align*}
By definition, we recall that
    \begin{align*}
      & \|{u}\|_{H^k(\R^2)}=\|(1+|\xi|^2)^{\frac{k}{2}}\hat{u}\|_{L^2(\R^2)},\quad k=0, 1, 2.
    \end{align*}
    For a fixed constant $R> 1,$ we have\begin{align*}
      \|{u}\|_{L^{\infty}(\R^2)}&\leq \frac{1}{2\pi}\|\hat{u}\|_{L^1(\R^2)}\\
      &\leq \|(1+|\xi|^2)^{\f12}(R+|\xi|^2)^{\f12}\hat{u}\|_{L^2(\R^2)}\|(1+|\xi|^2)^{-\f12}(R+|\xi|^2)^{-\f12}\|_{L^2(\R^2)}.
    \end{align*}
    We now compute that
    \begin{align*}
\|(1+|\xi|^2)^{-\f12}(R+|\xi|^2)^{-\f12}\|_{L^2(\R^2)}^2&=\int_{\R^2}(1+|\xi|^2)^{-1}(R+|\xi|^2)^{-1}d\xi\\
&=\int_0^{+\infty}2\pi(1+r^2)^{-1}(R+r^2)^{-1}rdr\\
&=\frac{\pi}{R-1}\ln\frac{1+r^2}{R+r^2}\Big|_{r=0}^{+\infty}=\frac{\pi\ln R}{R-1}.
\end{align*}
Moreover we estimate that
\begin{align*}
\|(1+|\xi|^2)^{\f12}(R+|\xi|^2)^{\f12}\hat{u}\|_{L^2(\R^2)}^2&=\|(1+|\xi|^2)\hat{u}\|_{L^2(\R^2)}^2+(R-1)\|(1+|\xi|^2)^{\f12}\hat{u}\|_{L^2(\R^2)}^2
\\&=\|{u}\|_{H^2(\R^2)}^2+(R-1)\|{u}\|_{H^1(\R^2)}^2.
\end{align*}
Therefore
\begin{align*}
      &\|{u}\|_{L^{\infty}(\R^2)}^2\leq (\|{u}\|_{H^2(\R^2)}^2+(R-1)\|{u}\|_{H^1(\R^2)}^2)\frac{\pi\ln R}{R-1}.
    \end{align*}
    Now we take $R=1+\|{u}\|_{H^2(\R^2)}/\|{u}\|_{H^1(\R^2)}$.
    Then
    \begin{align*}
      &\|{u}\|_{L^{\infty}(\R^2)}^2\leq 2\pi\|{u}\|_{H^1(\R^2)}^2\ln(1+\|{u}\|_{H^2(\R^2)}/\|{u}\|_{H^1(\R^2)}),
    \end{align*}
    which implies the first inequality on the whole space $\mathbb{R}^2$.

    Now for the second inequality on a finite region and the third inequality on the complement of finite region, fix an even smooth function $ {\psi}:\R\to[0,1]$, which is supported on $[-1,1]$ and equals to one in $[-\frac{5}{6}, \frac{5}{6}]$. Then define smooth functions
    $$ \eta_1(x):={\psi}(|x|),\quad  \eta_2(x):={\psi}(\frac{4}{3}|x|) $$
     such that $ \eta_2$ is supported in $B_{\frac{3}{4}} $ and $1- \eta_1$ is supported in $\overline{B_{\frac{5}{6}}^c} $.  Therefore
     \begin{align}\label{eq:1u}
      &\|\eta_2 u\|_{H^1(\R^2)}\leq C\| u\|_{H^1(B_{3/4})},\ \|\eta_2 u\|_{H^2(\R^2)}\leq C\| u\|_{H^2(B_{3/4})},\\ \label{eq:2u}&\|(1- \eta_1) u\|_{H^1(\R^2)}\leq C\| u\|_{H^1(B_{5/6}^c)},\ \|(1- \eta_1) u\|_{H^2(\R^2)}\leq C\| u\|_{H^2(B_{5/6}^c)}.
    \end{align}
    Here $C>1$ is a universal constant. As $ \eta_2=1$ in $B_{1/2} $, by \eqref{eq:1u} and the above logarithmic Sobolev inequality, we have
    \begin{align*}
      \|u\|_{L^\infty(B_{1/2})} &\leq \|\eta_2u\|_{L^\infty(\R^2)}\\
      & \leq C \|\eta_2 u\|_{H^1(\R^2)}\left(1+\ln \frac{\|\eta_2u\|_{H^2(\R^2)}}{\|\eta_2 u\|_{H^1(\R^2)}}\right)^{\f12}\\
      &\leq C \|\eta_2 u\|_{H^1(\R^2)}\left(1+\ln \frac{\| u\|_{H^2(B_{3/4})}}{\|\eta_2 u\|_{H^1(\R^2)}}\right)^{\f12}\\
      &\leq C  \| u\|_{H^1(B_{3/4})}\left(1+\ln \frac{\| u\|_{H^2(B_{3/4})}}{\| u\|_{H^1(B_{3/4})}}\right)^{\f12},
    \end{align*}
    which implies the second logarithmic Sobolev inequality. Here the last step follows from the fact that $ \delta\cdot(1+\ln(1/\delta))^{\f12} $ is increasing for $ \delta\in(0,1].$  The third Sobolev inequality on the complement of finite region follows in a similar way in view of  $\|u\|_{L^\infty(B_1^c)}\leq \|(1- \eta_1)u\|_{L^\infty(\R^2)} $.

    Finally for the improved Sobolev inequality on annulus, let $$\widetilde{u}(r,\theta)=u(r\cos\theta, r\sin\theta).$$
     Then
    \begin{align*}
      &\|\widetilde{u}\|_{H^k([1/4,3]\times[-\pi/2,\pi/2])}\leq C\|u\|_{H^k(\R^2)},\quad k=0,1,2,\\
      & \|\partial_{\theta}\widetilde{u}\|_{L^2([1/4,3]\times[-\pi/2,\pi/2])}\leq C\|\nabb u\|_{L^2(\R^2)}.
    \end{align*}
    Now we define functions $$ \eta_3(r,\theta):={\psi}(r-5/4){\psi}(2\theta/\pi),\quad u_1(r,\theta)=\eta_3(r,\theta)\widetilde{u}(r,\theta)$$
    on $\mathbb{R}^2$ with Cartesian coordinates $r, \theta\in\mathbb{R}$ (not polar coordinates for $u_1$).
    In particular $ \eta_3$ is smooth and supported in $[1/4,3]\times[-\pi/2,\pi/2] $. Moreover $ \eta_3$ is constant $1$ in $[1/2,3/2]\times[-\pi/3,\pi/3]$. These properties lead to the following relations
    \begin{align}
    \label{eq:uHk}
      &\|{u}_1\|_{H^k(\R^2)}\leq C\|\widetilde{u}\|_{H^k([1/4,3]\times[-\pi/2,\pi/2])}\leq C\|u\|_{H^k(\R^2)},\ k=0,1,2,\\
      & \label{eq:uL2} \|({u}_1,\partial_{\theta}u_1)\|_{L^2(\R^2)}\leq C\|(\widetilde{u},\partial_{\theta}\widetilde{u})\|_{L^2([1/4,3]\times[-\pi/2,\pi/2])}\leq C\|(u,\nabb u)\|_{L^2(\R^2)}.
    \end{align}
    For fixed constant $R\geq 1,$ let $u_R(r,\theta)=u_1(r/R,\theta R)$. We compute that
    \begin{align*}
       \|\partial_r{u}_R\|_{L^2(\R^2)}&=R^{-1}\|\partial_r{u}_1\|_{L^2(\R^2)},\quad \|\partial_{\theta}{u}_R\|_{L^2(\R^2)}=R\|\partial_{\theta}{u}_1\|_{L^2(\R^2)},\\
       R^{-1}\|{u}_1\|_{H^1(\R^2)} &\leq \|{u}_R\|_{H^1(\R^2)}\leq R^{-1}\|{u}_1\|_{H^1(\R^2)}+ R\|({u}_1,\partial_{\theta}u_1)\|_{L^2(\R^2)},\\
       \|{u}_R\|_{H^2(\R^2)}& \leq R^{2}\|{u}_1\|_{H^2(\R^2)}, \quad \|{u}_R\|_{L^2(\R^2)}=\|{u}_1\|_{L^2(\R^2)}.
    \end{align*}
    As a compactly supported function on $\mathbb{R}^2$, by using the logarithmic Sobolev embedding on the whole space, we have
    \begin{align*}
      \|u_1\|_{L^\infty(\R^2)}& =\|u_{R}\|_{L^\infty(\R^2)}\\
      &\leq  C \|u_{R}\|_{H^1(\R^2)}\left(1+\ln \frac{\|u_{R}\|_{H^2(\R^2)}}{\|u_{R}\|_{H^1(\R^2)}}\right)^{\f12}\\
      & \leq C (R^{-1}\|{u}_1\|_{H^1(\R^2)}+ R\|({u}_1,\partial_{\theta}u_1)\|_{L^2(\R^2)})\left(1+\ln \frac{R^2\|u_{1}\|_{H^2(\R^2)}}{R^{-1}\|u_{1}\|_{H^1(\R^2)}}\right)^{\f12},
    \end{align*}
    Now by taking $R=\|{u}_1\|_{H^1(\R^2)}^{\frac{1}{2}} \|({u}_1,\partial_{\theta}u_1)\|_{L^2(\R^2)}^{-\frac{1}{2}}$, we obtain that
    \begin{align*}
      &R^{-1}\|{u}_1\|_{H^1(\R^2)}=R\|({u}_1,\partial_{\theta}u_1)\|_{L^2(\R^2)}=\|{u}_1\|_{H^1(\R^2)}^{\frac{1}{2}}\|({u}_1,\partial_{\theta}u_1)\|_{L^2(\R^2)}^{\frac{1}{2}},\\
      &\ln \frac{R^2\|u_{1}\|_{H^2(\R^2)}}{R^{-1}\|u_{1}\|_{H^1(\R^2)}}\leq \frac{3}{2}\ln \frac{R^2\|u_{1}\|_{H^2(\R^2)}}{\|u_{1}\|_{H^1(\R^2)}}=\frac{3}{2}\ln \frac{\|u_{1}\|_{H^2(\R^2)}}{\|({u}_1,\partial_{\theta}u_1)\|_{L^2(\R^2)}}.
    \end{align*}
    Therefore from the previous inequality and estimates \eqref{eq:uHk}, \eqref{eq:uL2}, we conclude that
    \begin{align*}
      \|u_1\|_{L^\infty(\R^2)}^2 &\leq  C \|{u}_1\|_{H^1(\R^2)}\|({u}_1,\partial_{\theta}u_1)\|_{L^2(\R^2)}\left(1+\ln \frac{\|u_{1}\|_{H^2(\R^2)}}{\|({u}_1,\partial_{\theta}u_1)\|_{L^2(\R^2)}}\right)\\&\leq  C \|{u}\|_{H^1(\R^2)}\|(u,\nabb u)\|_{L^2(\R^2)}\left(1+\ln \frac{\|u\|_{H^2(\R^2)}}{\|(u,\nabb u)\|_{L^2(\R^2)}}\right).
    \end{align*}
    As  $ \eta_3=1$ in $[\frac{1}{2}, \frac{3}{2}]\times[-\frac{\pi}{3}, \frac{\pi}{3}]$, $u_1=\eta_3 \widetilde{u}$, we therefore have that
    \begin{align*}
      \|\widetilde{u}\|_{L^\infty([\frac{1}{2}, \frac{3}{2}]\times[-\frac{\pi}{3}, \frac{\pi}{3}])}^2 &=\|u_1\|_{L^\infty(\R^2)}^2 \\
      &\leq  C \|{u}\|_{H^1(\R^2)}\|(u,\nabb u)\|_{L^2(\R^2)}\left(1+\ln \frac{\|u\|_{H^2(\R^2)}}{\|(u,\nabb u)\|_{L^2(\R^2)}}\right).
    \end{align*}
    By symmetry (or considering $u_{(a)}(r\cos\theta, r\sin\theta)=u(r\cos(\theta+a), r\sin(\theta+a))$), we also have
\begin{align*}
   \|\widetilde{u}\|_{L^\infty([\frac{1}{2}, \frac{3}{2}]\times[a- \frac{\pi}{3}, a+\frac{\pi}{3}])}^2&\leq  C \|{u}\|_{H^1(\R^2)}\|(u,\nabb u)\|_{L^2(\R^2)}\left(1+\ln \frac{\|u\|_{H^2(\R^2)}}{\|(u,\nabb u)\|_{L^2(\R^2)}}\right),\ \forall\ a\in\R.
\end{align*}
In particular we derive that
\begin{align*}
   \|\widetilde{u}\|_{L^\infty([\frac{1}{2}, \frac{3}{2}]\times\R)}^2&\leq  C \|{u}\|_{H^1(\R^2)}\|(u,\nabb u)\|_{L^2(\R^2)}\left(1+\ln \frac{\|u\|_{H^2(\R^2)}}{\|(u,\nabb u)\|_{L^2(\R^2)}}\right),
\end{align*}
which implies the fourth logarithmic Sobolev inequality as $x=(r\cos\theta, r\sin\theta)\in A_{1,1/2}$ is equivalent to $(r, \theta)\in [\frac{1}{2}, \frac{3}{2}]\times [0, 2\pi]$ with $\widetilde{u}(r,\theta)=u(r\cos\theta, r\sin\theta)$.
\end{proof}

To apply the above logarithmic Sobolev inequality, we use the potential energy decay of the previous section to derive the following necessary bounds on fixed time $t$ by using the conformal energy identity.
\begin{Prop}
  \label{prop:PD:allp:Ats}
 The solution $\phi$ to \eqref{eq:NLW:semi:2d} verifies the following estimates for all $\ep>0$
  \begin{align}
 \label{eq:phi:L2}
 (1+t)^2\|\nabb \phi\|_{L^2}^2+\|(1+|t-r|)\nabla \phi\|_{L^2}^2+\|\phi\|_{L^2}^2&\leq C_{\ep}
 (1+t)^{\frac{5-p}{2}},\\
  \label{eq:phi:H2}
  \|\phi\|_{H^2}&\leq C_{\ep } 
  (1+t)^{1+\ep}
 \end{align}
 for some constant $C_{\ep}$ depending on $\ep$, the initial conformal energy $\mathcal{E}_{0, 2}$ and the first order energy $\mathcal{E}_{1, 0}$.
  \end{Prop}
\begin{proof}
  The proof relies on the conformal energy identity
  \begin{align*}
  &Q_0(t)+\frac{2}{p+1}\int_{\mathbb{R}^{2}} (t^2+r^2)|\phi(t, x)|^{p+1}dx +\frac{p-5}{p+1}\int_s^t 2\tau \int_{\mathbb{R}^2} |\phi(\tau, x)|^{p+1}dxd\tau \\
  &=Q_0(s)+\frac{2}{p+1}\int_{\mathbb{R}^2} (s^2+|x|^2)|\phi(s, x)|^{p+1}dx,
\end{align*}
which can be obtained by using the conformal Killing vector field as multiplier (see for example Lemma 1 in \cite{Pecher82:NLW:2d}). Here
\begin{align*}
  Q_0(t)= \int_{\mathbb{R}^2} |x_1\pa_2\phi-x_2\pa_1\phi|^2+|S\phi+ \phi|^2+\sum\limits_{j=1}^2|x_j\pa_t\phi+t\pa_j\phi|^2 dx
\end{align*}
with $S=t\pa_t+x_1\pa_1+x_2\pa_2$ the scaling vector field.

  Now by setting $s=0$ and using the time decay estimate of Proposition \ref{prop:timedecay}, we derive that
\begin{align*}
Q_0(t)&\leq C\mathcal{E}_{0, 2} +C\mathcal{E}_{0, 2} \int_0^t s(1+s)^{-\frac{p-1}{2}}ds\\
&\leq C\mathcal{E}_{0, 2}  (1+t)^{\frac{5-p}{2}}.
\end{align*}
Here and in the following $C$ denotes a universal constant.
On the other hand, we can estimate that
\begin{align}\label{eq:Q0}
\int_{\mathbb{R}^2}(t^2+r^2)|\nabb\phi|^2+|t\pa_r\phi+r\pa_t\phi|^2+|S\phi+\phi|^2  dx\leq C Q_0(t)\leq C\mathcal{E}_{0, 2}  (1+t)^{\frac{5-p}{2}} .
\end{align}
To control the $L^2$ norm of $\phi$, we note that
\begin{align*}
 &\int_{\mathbb{R}^2}(t\partial_t\phi+r\partial_r\phi+\phi)\phi dx=\int_{\mathbb{R}^2}t\partial_t\phi\cdot\phi dx=\frac{t}{2}\partial_t\int_{\mathbb{R}^2}|\phi|^2 dx.
\end{align*}
Therefore
 \begin{align*}
& \frac{t}{2}\frac{d}{dt}\|\phi(t)\|_{L^2}^2=\int_{\mathbb{R}^2}(t\partial_t\phi+r\partial_r\phi+\phi)\phi dx\leq \|\phi(t)\|_{L^2}\|(t\partial_t\phi+r\partial_r\phi+r\phi)(t)\|_{L^2}.
\end{align*}
By using the bound
\begin{align*}
\int_{\mathbb{R}^2}|S\phi+\phi|^2dx\leq Q_0(t)\leq C\mathcal{E}_{0, 2} (1+t)^{\frac{5-p}{2}},
\end{align*}
we conclude that
\begin{align}\nonumber
  \|\phi(t)\|_{L^2}\leq \|\phi(1)\|_{L^2}+C\sqrt{\mathcal{E}_{0, 2} } (t+1)^{\frac{5-p}{4}},\quad \forall\ t\geq 1.
\end{align}
We also have
\begin{align}
\nonumber
  \|\phi(s)\|_{L^2}\leq \|\phi(0)\|_{L^2}+\int_0^1\|\partial_t\phi(t)\|_{L^2}dt,\quad \forall\ s\in[0,1]
\end{align}
with
 $$\|\phi(0)\|_{L^2}=\|\phi_0\|_{L^2}\leq C\sqrt{\mathcal{E}_{0, 2} }$$
 by using Hardy's inequality. This together with the standard energy conservation
\begin{align}\label{eq:E0}
\int_{\mathbb{R}^2}|\pa_t\phi|^2+|\nabla\phi|^2+\frac{2}{p+1}|\phi|^{p+1}dx=\mathcal{E}_{0, 0} ,
\end{align}leads to\begin{align}\label{eq:phit}
  \|\phi(t)\|_{L^2}\leq C\sqrt{\mathcal{E}_{0, 2} } (t+1)^{\frac{5-p}{4}},\ \forall\ t\geq 0.
\end{align}
Therefore we can show that
\begin{align*}
&\int_{\mathbb{R}^2}(t-r)^2|\nabla\phi|^2dx \\
&\leq C\int_{\mathbb{R}^2} |(t+r)(\pa_t+\pa_r)\phi+\phi|^2+|(t-r)(\pa_t-\pa_r)\phi+\phi|^2+(t^2+r^2)|\nabb\phi|^2+|\phi|^2dx\\
&\leq C \int_{\mathbb{R}^2}(t^2+r^2)|\nabb\phi|^2+|t\pa_r\phi+r\pa_t\phi|^2+|S\phi+\phi|^2 +|\phi|^2 dx\\
&\leq C\mathcal{E}_{0, 2} (1+t)^{\frac{5-p}{2}}.
\end{align*}
Estimate \eqref{eq:phi:L2} then follows from the above bounds \eqref{eq:Q0}, \eqref{eq:E0}, \eqref{eq:phit}.

\bigskip

Finally for the $H^2$ estimate \eqref{eq:phi:H2} for $\phi$, note that
\begin{align*}
\|\phi\|_{H^1}\leq C\sqrt{\mathcal{E}_{0, 2}} (1+t)^{\frac{5-p}{4}}+C\sqrt{\mathcal{E}_{0, 0}}\leq C\sqrt{\mathcal{E}_{0, 2}} (1+t)
\end{align*}
as $p>1$. It remains to estimate $\|\nabla^2\phi\|_{L^2}$.
For any $0<\ep<\frac{p-1}{p+1}$, by using Gagliardo-Nirenberg inequality, we can show that
\begin{align*}
\||\nabla\phi|\cdot |\phi|^{p-1}\|_{L^2} &\leq \|\nabla\phi\|_{L^{2+\ep}}\|\phi\|_{L^{\frac{2(2+\ep)(p-1)}{\ep}}}^{p-1}\\
&\leq C_{\ep}\|\nabla\phi\|_{L^2}^{\frac{2}{2+\ep}}\|\nabla^2\phi\|_{L^2}^{\frac{\ep}{2+\ep}} \|\phi\|_{L^{p+1}}^{\frac{(p+1)\ep}{2(2+\ep)}}\|\nabla\phi\|_{L^2}^{p-1-\frac{(p+1)\ep}{2(2+\ep)}}\\
&\leq C_{\ep} \mathcal{E}_{0, 0}^{\frac{p}{2}-\frac{(p+1)\ep}{4(2+\ep)}}\|\nabla^2\phi\|_{L^2}^{\frac{\ep}{2+\ep}}.
\end{align*}
Here we have used the energy conservation to control $\|\nabla\phi\|_{L^2}$ and $\|\phi\|_{L^{p+1}}$. Then in view of energy estimate we derive that
\begin{align*}
\|\nabla^2\phi\|_{L^2}&\leq \sqrt{\mathcal{E}_{1, 0} }+p\int_0^{t}\||\nabla\phi|\cdot |\phi|^{p-1}\|_{L^2}ds\\
&\leq \sqrt{\mathcal{E}_{1, 0} }+ C_{\ep} \mathcal{E}_{0, 0}^{\frac{p}{2}-\frac{(p+1)\ep}{4(2+\ep)}}\int_0^t\|\nabla^2\phi\|_{L^2}^{\frac{\ep}{2+\ep}} ds,
\end{align*}
from which we derive that
\begin{align*}
\|\nabla^2\phi\|_{L^2}\leq C\sqrt{\mathcal{E}_{1, 0}}+C_{\ep}\mathcal{E}_{0, 0}^{\frac{p}{2}+\frac{(p-1)\ep}{8}}(1+t)^{1+\frac{\ep}{2}}.
\end{align*}
Replacing $\ep$ by $2\ep$,  we conclude estimate \eqref{eq:phi:H2}.

\end{proof}

As a consequence of the above Proposition and the logarithmic Sobolev inequality, we derive the pointwise decay estimate for the solution for all $1<p<5$.
\begin{Prop}
\label{prop:ptdecay:allp}
For any $\ep>0$, the solution $\phi$ of \eqref{eq:NLW:semi:2d} verifies the following pointwise decay estimate
\begin{align*}
|\phi|\leq C_{\ep}
(1+t+|x|)^{-\frac{p-1}{8}+\ep}
\end{align*}
for some constant $C_{\ep}$ depending on $\ep$, the initial conformal energy $\mathcal{E}_{0, 2}$ and the initial first order energy $\mathcal{E}_{1, 0}$.

The proof also implies that the solution decays better away from the light cone, that is,
\begin{align*}
|\phi|\leq C_{\ep}
(1+t)^{-\frac{p-1}{4}+\ep}, \quad |t-r|\geq \frac{t}{2}.
\end{align*}
\end{Prop}
\begin{proof}
The estimate trivially holds when $t+|x|\leq 10$. In the sequel, let's assume without loss of generality that  $t+|x|>10$. Divide $\mathbb{R}^2$ into three regions: $B_{\frac{t}{2}}$, $B_{\frac{3t}{2}}^c$, $\{\frac{t}{2}\leq |x|\leq \frac{3t}{2}\}$. Here recall that $B_{R}$ denotes the spatial ball with radius $R$.

On region $B_{\frac{t}{2}}$, for fixed $t>1$, define
\begin{align}\label{deftphi}
\tilde{\phi}(x)=\phi(t, tx),\ \text{for}\ x\in\R^2.
\end{align}
In view of the decay estimate \eqref{eq:phi:L2}, we conclude that
\begin{align*}
\int_{|x|\leq 3t/4}|\nabla\phi|^2dx\leq 4^2t^{-2}\int_{|x|\leq t} (t-r)^2 |\nabla\phi|^2 dx \les 
(1+t)^{\frac{1-p}{2}}.
\end{align*}
Thus we can compute that
\begin{align*}
\int_{|x|\leq 3/4}|\tilde{\phi}|^2+|\nabla\tilde{\phi}|^2dx \leq C\int_{|x|\leq 3t/4}( t^{-2}|\phi|^2 +|\nabla\phi|^2) dx\les
(1+t)^{\frac{1-p}{2}}.
\end{align*}
Similarly for the second order derivative
\begin{align*}
\int_{|x|\leq 3/4}|\nabla^2 \tilde{\phi}|^2 dx=\int_{|x|\leq 3t/4} t^2 |\nabla^2\phi|^2 dx \les  t^2+ (1+t)^{4+\ep}\les (1+t)^{4+\ep}.
\end{align*}
Hence from the logarithmic Sobolev embedding of Lemma \ref{lem:log:Sobolev}, we derive that on the region $\{|x|\leq \frac{t}{2}\}$, the solution satisfies
\begin{align*}
|\phi|\leq |\tilde{\phi}|\les (1+t)^{\frac{1-p}{4}+\frac{\ep}{2}}\sqrt{\ln (C_{\ep}t)}\les (1+t)^{\frac{1-p}{4}+\ep}.
\end{align*}
Here by using our notation the implicit constant $C_{\ep}$ relies on $\ep$, the initial data $\mathcal{E}_{0, 2}$ and $\mathcal{E}_{1, 0}$.

On the region $B_{\frac{3t}{2}}$, fix a point $(t, x_0)$ with $r_0=|x_0|\geq \frac{3t}{2}$. Define
\begin{align*}
\psi(x)=\phi(t, r_0x),\quad |x|\geq 1.
\end{align*}
Similarly we can show that
\begin{align*}
\int_{|x|\geq 5/6} |\psi|^2+|\nabla\psi|^2 &\leq 6^2 r_0^{-2}\int_{\mathbb{R}^2} |\phi|^2 +|(r-t)\nabla\phi|^2 dx\\
&\les 
r_0^{-2} (1+t)^{\frac{5-p}{2}}\\
&\les
 (1+r_0)^{\frac{1-p}{2}}.
\end{align*}
Here as we have assumed that $t+|x_0|\geq 10$ and $|x_0|\geq \frac{3}{2}t$.
For the second order derivative of $\psi$, we have
\begin{align*}
\int_{|x|\geq 5/6}|\nabla^2 \psi|^2 dx=r_0^2 \int_{|x|\geq 5r_0/6} |\nabla^2 \phi|^2 dx\les  
r_0^2+
(1+r_0)^{4+\ep}\les (1+r_0)^{4+\ep}.
\end{align*}
Therefore according the logarithmic Sobolev embedding, we derive that
\begin{align*}
|\phi(t, x_0)|\les 
(1+r_0)^{\frac{1-p}{4}+\ep},\quad r_0=|x_0|\geq \frac{3}{2}t.
\end{align*}
Finally on the region near the light cone, consider the annulus $A_{t, t/2}$. For fixed $t>2$, we still define $\widetilde{\phi} $ as in \eqref{deftphi}.
 Therefore in view of the estimates in Proposition \ref{prop:PD:allp:Ats}, we can show that
\begin{align*}
\|\widetilde{\phi}\|_{L^2(\R^2)}=t^{-1}\|{\phi}(t)\|_{L^2(\R^2)} &\les 
(1+t)^{\frac{1-p}{4}},\\
 \|\nabla\widetilde{\phi}\|_{L^2(\R^2)}=\|\nabla{\phi}(t)\|_{L^2(\R^2)} &\les 1, \\ 
\|\nabb\widetilde{\phi}\|_{L^2(\R^2)}=\|\nabb{\phi}(t)\|_{L^2(\R^2)} &\les 
(1+t)^{\frac{1-p}{4}},\\
 \|\widetilde{\phi}\|_{H^2(\R^2)}\leq t\|{\phi}(t)\|_{H^2(\R^2)} &\les 
 t+
 (1+t)^{2+\frac{\ep}{2}}\les (1+t)^{2+\frac{\ep}{2}}.
\end{align*}
Therefore by using the logarithmic Sobolev embedding adapted to the annulus $A_{1, 1/2}$ for $\widetilde{\phi}$, we derive that
\begin{align*}
 \|\phi(t)\|_{L^{\infty}(A_{t,t/2})}^2
&=\|\widetilde{\phi}\|_{L^{\infty}(A_{1,1/2})}^2\\
& \leq C \|\widetilde{\phi}\|_{H^1(\R^2)}\|(\widetilde{\phi},\nabb \widetilde{\phi})\|_{L^2(\R^2)}\left(1+\ln\frac{\|\widetilde{\phi}\|_{H^{2}(\R^2)}} {\|(\widetilde{\phi},\nabb \widetilde{\phi})\|_{L^2(\R^2)}}\right) \\
&\les
(1+t)^{\frac{1-p}{4}}\left(1+\ln t \right)\\
&\leq C_{\ep} 
(1+t)^{\frac{1-p}{4}+2\ep }.
\end{align*}
By our notation the implicit constant as well as $C_{\ep}$ depends on $\ep>0$, the initial data $\mathcal{E}_{0, 2}$ and $\mathcal{E}_{1, 0}$.
We thus finished the proof for Proposition \ref{prop:ptdecay:allp}.
\end{proof}
As a corollary of the above pointwise decay estimate, the solution scatters in energy space when $p>2\sqrt{5}-1$.
\begin{Cor}
When $p>2\sqrt{5}-1$, the solution $\phi$ to \eqref{eq:NLW:semi:2d} verifies the following spacetime bound
\begin{align*}
\|\phi\|_{L_t^p L_x^{2p}}\leq C 
\end{align*}
for some constant $C$ depending on $p$, the initial data $\mathcal{E}_{0, 2}$ and $\mathcal{E}_{1, 0}$.
Consequently the solution scatters in energy space, that is, there exists pairs $(\phi_0^{\pm}(x), \phi_1^{\pm}(x))$ such that
 \begin{align*}
      \lim\limits_{t\rightarrow\pm\infty} \| \pa \phi(t,x)-\pa \mathbf{L}(t)(\phi_0^{\pm}(x), \phi_1^{\pm}(x))\|_{ L_x^2}=0.
\end{align*}
Here $\mathbf{L}(t)$ is the linear evolution operator for wave equation.
\end{Cor}
\begin{proof}
Without loss of generality we only consider the spacetime norm in the future $t\geq 0$.
In view of the potential energy decay estimate of Proposition \ref{prop:timedecay} as well as the pointwise decay estimate of the above proposition, we can show that
\begin{align*}
\|\phi\|_{L_t^{p} L_x^{2p}}^p &=\int_{0}^{\infty}\left(\int_{\mathbb{R}^{2}}|\phi|^{p+1}|\phi|^{p-1}dx\right)^{\frac{1}{2}}dt\\
&\les \int_{0}^{\infty}\left(
(1+t)^{-\frac{p-1}{2}} ( 
(1+t)^{-\frac{p-1}{8}+\ep })^{p-1}\right)^{\frac{1}{2}}dt\\
&\les
  \int_{0}^{\infty}(1+t)^{-\frac{(p-1)(p+3)}{16}+\frac{p-1}{2}\ep }dt.
\end{align*}
Since $p>2\sqrt{5}-1$, we in particular have that
\begin{align*}
(p-1)(p+3)> 16.
\end{align*}
Choose $\ep$ sufficiently small such that
\begin{align*}
-\frac{(p-1)(p+3)}{16}+\frac{p-1}{2}\ep <-1.
\end{align*}
We then conclude that
\begin{align*}
\|\phi\|_{L_t^pL_x^{2p}}\leq C_{\ep}
\end{align*}
for some constant $C_{\ep}$ depending only on $\ep>0$ (which could be fixed), $p$ and the initial data $\mathcal{E}_{0, 2}$, $\mathcal{E}_{1, 0}$. The scattering result follows by a standard argument, see for example \cite{Pecher82:NLW:2d}.
\end{proof}

\subsection{Improved sharp decay estimates for larger $p$}
Since the solution decays in time, the nonlinearity decays faster for larger $p$. It is believed that for sufficiently large $p$, the solution to the nonlinear wave equation \eqref{eq:NLW:semi:2d} decays as fast as linear waves (for example the superconformal case $p\geq 5$ addressed in \cite{Pecher82:NLW:2d}). We conjecture that for equation \eqref{eq:NLW:semi:2d} on $\mathbb{R}^{1+2}$, this holds when $p>3$ (that is for sufficiently smooth and localized data the solution decays like $t^{-\frac{1}{2}}$,  which is stronger than that obtained in the previous section).
In this section, we show that this conjecture holds when $p>\frac{11}{3}$.

For fixed $t\geq 0,\ r>0$ define\begin{align}\label{defphi*}
&\phi_*(t, r)=\sup\{|\phi(t, x)|:|x|=r\}.
\end{align}
\begin{Lem}
\label{lem:decay}
For the solution $\phi$ of \eqref{eq:NLW:semi:2d}, let $\phi_*$ be defined as above in \eqref{defphi*}.
Then for $1<p<5$ and  $t\geq 0$
\begin{align*}
&\int_0^{+\infty}(1+t+r)^{\frac{p-1}{2}}|\phi_*(t, r)|^{\frac{p+3}{2}}dr\leq C
\end{align*}
for some constant $C$ depending only on $\mathcal{E}_{0, 2}$ and $\mathcal{E}_{1, 0}$.
\end{Lem}
\begin{proof}
For fixed $t\geq 0,\ r>0$ define
\begin{align*}
\widetilde{\phi}(t, r,\theta)&=\phi(t, r\cos\theta, r\sin\theta),\\
 {\phi}_-(t, r)&=\inf\{|\phi(t, x)|:|x|=r\}=\inf\{|\widetilde{\phi}(t, r,\theta)|:\theta\in\R\},\\
 A_1(t, r)&=\frac{1}{2\pi}\int_{-\pi}^{\pi}\widetilde{\phi}(t, r,\theta)d\theta,\\
  A_2(t, r)&=\int_{-\pi}^{\pi}|\partial_{\theta}\widetilde{\phi}(t, r,\theta)|^2d\theta,\\
   A_3(t, r)&=\int_{-\pi}^{\pi}|\widetilde{\phi}(t, r,\theta)|^{p+1}d\theta.
\end{align*}
Then we have  $0\leq{\phi}_-(t, r)\leq |A_1(t, r)|$ (if $ \phi$ is real valued). By the definition of $\phi_*,\ {\phi}_-,\ \widetilde{\phi},\ A_1,\ A_2,\\ A_3$ and H\"older inequality, we have
\begin{align*}
\phi_*^{\frac{p+3}{2}}(t, r)-\phi_-^{\frac{p+3}{2}}(t, r)&=\sup_{\theta\in[-\pi,\pi)}|\widetilde{\phi}(t, r,\theta)|^{\frac{p+3}{2}}-\inf_{\theta\in[-\pi,\pi)}|\widetilde{\phi}(t, r,\theta)|^{\frac{p+3}{2}}\\
& \leq \int_{-\pi}^{\pi}|\partial_{\theta}(|\widetilde{\phi}|^{\frac{p+3}{2}})(t, r,\theta)|d\theta\\
&=\frac{p+3}{2}\int_{-\pi}^{\pi}|\partial_{\theta}\widetilde{\phi}||\widetilde{\phi}|^{\frac{p+1}{2}}(t, r,\theta)d\theta\\
&\leq \frac{p+3}{2}|A_2(t, r)A_3(t, r)|^{\f12}.
\end{align*}
In particular we obtain that
\begin{align}
\label{eq:phi-}
&\phi_*^{\frac{p+3}{2}}(t, r)\leq \phi_-^{\frac{p+3}{2}}(t, r)+C|A_2(t, r)A_3(t, r)|^{\f12}\leq |A_1(t, r)|^{\frac{p+3}{2}}+C|A_2(t, r)A_3(t, r)|^{\f12}.
\end{align}
Now using polar coordinate $ x=r(\cos\theta, \sin\theta)$ and the definition of $A_1,A_2,A_3$,
we write that
\begin{align*}
    \int_{\mathbb{R}^2}|\phi(t, x)|^{p+1}(1+t+r)^{\frac{p-1}{2}}dx&=\int_0^{+\infty}A_3(t, r)(1+t+r)^{\frac{p-1}{2}}rdr,
\\
\int_{\mathbb{R}^2}(1+t^2+r^2)|\nabb\phi|^2 dx&=\int_0^{+\infty}\int_{-\pi}^{\pi}(1+t^2+r^2)|r^{-1}\partial_{\theta}\widetilde{\phi}(t, r,\theta)|^2rd\theta dr\\
&=\int_0^{+\infty}A_2(t, r)(1+t^2+r^2)r^{-1}dr,\\
\int_{\mathbb{R}^2}|\phi|^2 dx &=\int_0^{+\infty}\int_{-\pi}^{\pi}|\widetilde{\phi}(t, r,\theta)|^2rd\theta dr
\geq2\pi\int_0^{+\infty}|A_1(t, r)|^2rdr,\\
\|(1+|t-r|)\nabla \phi\|_{L^2}^2 &\geq \|(1+|t-r|)\partial_r \phi\|_{L^2}^2\\
&=\int_0^{+\infty}\int_{-\pi}^{\pi}(1+|t-r|)^2|\partial_r\widetilde{\phi}(t, r,\theta)|^2rd\theta\\
& \geq2\pi\int_0^{+\infty}(1+|t-r|)^2|\partial_rA_1(t, r)|^2rdr.
\end{align*}
In view of the time decay of Proposition \ref{prop:timedecay}, we conclude that
\begin{align}\label{eq:A3}
    &\int_0^{+\infty}A_3(t, r)(1+t+r)^{\frac{p-1}{2}}rdr\leq C\mathcal{E}_{0, 2}.
\end{align}
By using estimates \eqref{eq:Q0} and \eqref{eq:E0}, we derive that
\begin{align}
    \nonumber
    \int_{\mathbb{R}^2}(1+t^2+r^2)|\nabb\phi|^2 dx &\leq  C\mathcal{E}_{0, 2}  (1+t)^{\frac{5-p}{2}},\\
     \nonumber
     \int_0^{+\infty}A_2(t, r)(1+t^2+r^2)r^{-1}dr &\leq C\mathcal{E}_{0, 2}  (1+t)^{\frac{5-p}{2}},\\
      \label{eq:A2}
      \int_0^{+\infty}A_2(t, r)(1+t+r)^{\frac{p-1}{2}}r^{-1}dr &\leq C(1+t)^{\frac{p-1}{2}-2}\int_0^{+\infty}A_2(t, r)(1+t^2+r^2)r^{-1}dr\\
      \nonumber
      &\leq C\mathcal{E}_{0, 2}.
\end{align}
Similarly estimate \eqref{eq:phi:L2} of Proposition \ref{prop:PD:allp:Ats} leads to
\begin{align}
\label{eq:A11}
    \int_0^{+\infty}(|A_1(t, r)|^2+(1+|t-r|)^2|\partial_rA_1(t, r)|^2)rdr &\les 
    (1+t)^{\frac{5-p}{2}},\\
    \label{eq:A12}
    \int_{\frac{t+1}{4}}^{+\infty}r^{\frac{p-1}{2}-1}|A_1(t, r)|^2dr &\leq C(1+t)^{\frac{p-1}{2}-2}\int_0^{+\infty}|A_1(t, r)|^2rdr\\
    \nonumber
    & \les 1.
\end{align}
Therefore from the above estimates \eqref{eq:A2} and  \eqref{eq:A3},
 we show that
 \begin{align}
 \label{eq:A2A3}
    \int_0^{+\infty}|A_2(t, r)A_3(t, r)|^{\f12}(1+t+r)^{\frac{p-1}{2}}dr&\leq\int_0^{+\infty}(r^{-1}A_2(t, r)+rA_3(t, r))(1+t+r)^{\frac{p-1}{2}}dr\\
    \nonumber&\leq 1. 
\end{align}
By using H\"older's inequality and the definition of $A_1$ and $A_3$, we have
\begin{align*}
 2\pi|A_1(t, r)|^{p+1}\leq A_3(t, r).
\end{align*}
We thus can show that
\begin{align}
\label{eq:A13}
    &|A_1(t, r)|^{\frac{p+3}{2}}=|A_1(t, r)|^{\frac{p+1}{2}}|A_1(t, r)|\leq |A_3(t, r)|^{\frac{1}{2}}|A_1(t, r)|\leq r|A_3(t, r)|+r^{-1}|A_1(t, r)|^2,\\ \label{eq:A14}&|\partial_r(|A_1(t, r)|^{\frac{p+3}{2}})|=\frac{p+3}{2}|A_1(t, r)|^{\frac{p+1}{2}}|\partial_rA_1(t, r)|\leq C|A_3(t, r)|^{\frac{1}{2}}|\partial_rA_1(t, r)|.
\end{align}
From estimates \eqref{eq:A13}, \eqref{eq:A3} and \eqref{eq:A12}, we conclude that
\begin{align}
\label{eq:A1r>t}
    &\int_{\frac{t+1}{4}}^{+\infty}r^{\frac{p-1}{2}}|A_1(t, r)|^{\frac{p+3}{2}}dr\leq\int_{\frac{t+1}{4}}^{+\infty}r^{\frac{p-1}{2}}(r|A_3(t, r)|+r^{-1}|A_1(t, r)|^2)dr\les 1.
\end{align}
In view of bounds \eqref{eq:A14}, \eqref{eq:A3} and \eqref{eq:A11}, we show that
\begin{align*}
    &\int_{0}^{t+1}|A_1(t, r)|^{\frac{p+3}{2}}(t+1-2r)dr\\
    &=-\int_{0}^{t+1}r(t+1-r)\partial_r(|A_1(t, r)|^{\frac{p+3}{2}})dr\\
    &\leq  C\int_{0}^{t+1}r(t+1-r)|A_3(t, r)|^{\frac{1}{2}}|\partial_rA_1(t, r)|dr\\
    &\leq C\int_{0}^{t+1}(t+1)|A_3(t, r)|rdr+C(t+1)^{-1}\int_{0}^{t+1}(t+1-r)^2|\partial_rA_1(t, r)|rdr\\
    & \leq C(t+1)^{1-\frac{p-1}{2}}\int_0^{+\infty}A_3(t, r)(1+t+r)^{\frac{p-1}{2}}rdr+C(t+1)^{-1}\mathcal{E}_{0, 2}(1+t)^{\frac{5-p}{2}}\\
    & \les
    (1+t)^{\frac{3-p}{2}},
\end{align*}
which together with \eqref{eq:A1r>t} implies that
\begin{align*}
    &\int_{0}^{\frac{t+1}{2}}|A_1(t, r)|^{\frac{p+3}{2}}(t+1-2r)dr\\
    &=\int_{0}^{t+1}|A_1(t, r)|^{\frac{p+3}{2}}(t+1-2r)dr+\int_{\frac{t+1}{2}}^{t+1}|A_1(t, r)|^{\frac{p+3}{2}}(2r-t-1)dr\\
    & \les (1+t)^{\frac{3-p}{2}}+ (1+t)^{\frac{3-p}{2}}\int_{\frac{t+1}{2}}^{t+1}|A_1(t, r)|^{\frac{p+3}{2}}r^{\frac{p-1}{2}}dr\\
    &\les
    (1+t)^{\frac{3-p}{2}}.
\end{align*}
Moreover for small $r$, we bound that
\begin{align}
\label{eq:A1r<t}
    &\int_0^{\frac{t+1}{4}}|A_1(t, r)|^{\frac{p+3}{2}}dr\leq \frac{2}{t+1}\int_{0}^{\frac{t+1}{2}}|A_1(t, r)|^{\frac{p+3}{2}}(t+1-2r)dr\les (1+t)^{\frac{1-p}{2}}.
\end{align}
Combining estimates \eqref{eq:A1r>t} and \eqref{eq:A1r<t}, we deduce that
\begin{align*}
    &\int_0^{+\infty}|A_1(t, r)|^{\frac{p+3}{2}}(1+t+r)^{\frac{p-1}{2}}dr\les 1,
\end{align*}
which together with estimates \eqref{eq:phi-} and \eqref{eq:A2A3}
implies the Lemma. Here recall that by our notation the implicit constant relies only on the initial data $\mathcal{E}_{0, 2}$ and $\mathcal{E}_{1, 0}$.
\end{proof}

We also need the following integration lemma.
\begin{Lem}
\label{lem:log2}
For real constant $ A\neq 1$, let
\begin{align*}
&F(A)=\int_{\theta\in [-\pi,\pi),A+\cos\theta>0}\frac{1}{\sqrt{A+\cos\theta}}d\theta.
\end{align*}
Then we can bound that
\[
 F(A)\leq C(1+\ln(1+1/|A-1|))
 \]
 for some constant $C$ independent of $A$.
\end{Lem}
\begin{proof}
If $A\leq -1$ then $F(A)=0$. For the case when  $-1<A<1$,
we first can write that
\begin{align*}
 F(A)&=2 \int_{\theta\in [0,\pi),A+\cos\theta>0}\frac{1}{\sqrt{A+\cos\theta}}d\theta \\
 &=2\int_{-A}^1\frac{1}{\sqrt{A+z}\sqrt{1-z^2}}dz\\
 &=4\int_0^{\pi/2}\frac{1}{\sqrt{(1+A)\sin^2s+(1-A)}}ds.
\end{align*}
Now for the case when $-1<A\leq 0$, it trivially has that
\[
 (1+A)\sin^2s+(1-A)\geq (1-A)\geq 1.
 \]
 In particular
 \begin{align*}
 F(A)=&4\int_0^{\pi/2}\frac{1}{\sqrt{(1+A)\sin^2s+(1-A)}}ds\leq 2\pi.
\end{align*}
For the case when $0\leq A<1$,
we estimate that
\begin{align*}
(1+A)\sin^2s+(1-A)\geq \sin^2s+(1-A)\geq \frac{1}{2} (\frac{2s}{\pi}+\sqrt{1-A})^2,\quad \forall 0\leq s\leq\frac{\pi}{2}.
\end{align*}
Therefore we can bound that
\begin{align*}
 F(A)&\leq 4\sqrt{2}\int_0^{\pi/2}\frac{1}{\frac{2s}{\pi}+\sqrt{1-A}}ds=2\sqrt{2}\pi\ln\frac{1+\sqrt{1-A}}{\sqrt{1-A}}\leq 2\sqrt{2}\pi (1+\ln(1+1/|A-1|)).
\end{align*}
Similarly if $A>1$, we change variable as follows
\begin{align*}
 F(A)&=2 \int_{0}^{\pi}\frac{1}{\sqrt{A+\cos\theta}}d\theta \\
 &=4\int_0^{\pi/2}\frac{1}{\sqrt{2\sin^2s+(A-1)}}ds\\
 &\leq  4\pi \sqrt{2}\int_0^{\pi/2}\frac{1}{2\sqrt{2}s +\pi \sqrt{A-1}}ds\\
 &=2\pi\ln\frac{\sqrt{2}+\sqrt{1-A}}{\sqrt{1-A}}\\
 &\leq 2\sqrt{2} \pi(1+\ln(1+1/|A-1|)).
\end{align*}
This completes the proof.
\end{proof}

We will also use the following lemma.
\begin{Lem}
\label{lem:log1}
Let  $g$ be nonnegative function defined on $r>0$. Then for any $s>0$ and $x\in\R$, it holds
\begin{align*}
\int_{r>0}g(r)(1+\ln(1+2r |x-r|^{-1}))dr
&\leq C\int_{r>0}g(r)(1+\ln(1+s+r))dr+C\|g\|_{L^{\infty} }(1+s)^{-1}
\end{align*}
for some constant $C$ which is independent of $s$ and $x$.
\end{Lem}
\begin{proof}
Denote
 $$ \delta=(1+|x|+s)^{-1}.$$
 When $|x-r|\geq 1,\ r>0$, we bound that
\begin{align*}
1+2r \min (\delta^{-1},|x-r|^{-1})\leq 1+2r\leq (1+r)^2\leq 2(1+s+r)^2.
\end{align*}
Otherwise
if $|x-r|\leq 1,\ r>0$ (in particular $|x|\leq 1+r$), we have
\begin{align*}
1+2r \min(\delta^{-1},|x-r|^{-1}) \leq 1+2r(1+|x|+s)\leq 1+2r(2+r+s)\leq 2(1+s+r)^2.
\end{align*}
 Thus for $r>0$, we always have
 \begin{align*}
 1+2r \min(\delta^{-1},|x-r|^{-1})\leq 2(1+s+r)^2.
 \end{align*}
  Note that
  \begin{align*}
  &1+2r |x-r|^{-1}\leq (1+2r \min(\delta^{-1},|x-r|^{-1}))\cdot \max(\delta |x-r|^{-1},1).
  \end{align*}
  We therefore can show that
  \begin{align*}
&\int_{r>0}g(r)(1+\ln(1+2r |x-r|^{-1}))dr
\\
&\leq \int_{r>0}g(r)(1+\ln(1+2r \min(\delta^{-1},|x-r|^{-1})))dr+\int_{r>0}g(r)\ln\max(\delta |x-r|^{-1},1)dr\\
&\leq \int_{r>0}g(r)(1+\ln(2(1+s+r)^2))dr+\|g\|_{L^{\infty} }\int_{\R}\ln\max(\delta |x-r|^{-1},1)dr\\
&\leq 2\int_{r>0}g(r)(1+\ln(1+s+r))dr+2\|g\|_{L^{\infty} }(1+s)^{-1}.
\end{align*}
Here for the last step, we used
\begin{align*}
&\int_{\R}\ln\max(\delta/|x-r|,1)dr =2 \int_{0}^{\delta}\ln (\delta r^{-1}) dr=2\delta\leq 2(1+s)^{-1}.
\end{align*}
We thus finished the proof for the Lemma.
\end{proof}

With the above preparations, we are now ready to establish the following improved decay estimates for the solution for $p>\frac{11}{3}$, which together with Proposition 4.2 implies Theorem \ref{thm:main:pd}.
\begin{Prop}
When $\frac{11}{3}<p<5$, the solution to \eqref{eq:NLW:semi:2d} verifies the following sharp time decay estimate
\begin{align*}
|\phi(t, x)|\leq C(1+t+|x|)^{-\frac{1}{2}}
\end{align*}
for some constant $C$ depending only on $p,\ \mathcal{E}_{0, 2},\ \mathcal{E}_{1, 0}$.
\end{Prop}
\begin{proof}
Let's fix time $t>0$. Define
\begin{align*}
f(t)=\sup\limits_{x\in\mathbb{R}^2} |\phi(t, x)(1+t)^{\frac{1}{2}}|,\quad f_*(t)=\sup\limits_{x\in\mathbb{R}^2} |\phi(t, x)(1+t+|x|)^{\frac{1}{2}}|.
\end{align*}
In view of the rough decay estimate of Proposition \ref{prop:ptdecay:allp}, it is obvious that $f(t)$ is finite for all $t>0$. By the definition of $\phi_*$ (in \eqref{defphi*}) and $f$, we have
\begin{align*}
\phi_*(t, r)\leq f(t)(1+t)^{-\frac{1}{2}},
\end{align*}
which together with Lemma \ref{lem:decay} implies that
\begin{align}
\label{eq:trp}
\int_0^{+\infty}(1+t+r)^{\frac{p-1}{2}}|\phi_*(t, r)|^{p}dr&\leq \left|\sup_{r>0}\phi_*(t, r)\right|^{\frac{p-3}{2}}\int_0^{+\infty}(1+t+r)^{\frac{p-1}{2}}|\phi_*(t, r)|^{\frac{p+3}{2}}dr\\
\nonumber&\les |f(t)(1+t)^{-\frac{1}{2}}|^{\frac{p-3}{2}}.
\end{align}
Using the rough decay estimate of Proposition \ref{prop:ptdecay:allp} again and in view of the definition of $\phi_*$, we have\begin{align*}
\phi_*(t, r) &\les 
(1+t+r)^{-\frac{p-1}{8}+\ep},
\\
 r^{\frac{1}{2}}\phi_*^p(t, r) &\les
(1+t+r)^{-\frac{p(p-1)}{8}+\frac{1}{2}+p\ep}.
\end{align*}
For the case when $\frac{11}{3}<p<5$, we can choose $ \epsilon$ such that $$0<\ep<\frac{p-1}{8}-\frac{1}{p}.$$
In particular we conclude that
\begin{align}
\label{eq:rp}
r^{\frac{1}{2}}\phi_*^p(t, r)\les
(1+t+r)^{-\frac{1}{2}}.
\end{align}
Now for any point $(t_0, x_0)\in\mathbb{R}^{1+2}$, $t_0\geq 0,$ denote $r_0=|x_0|$. In view of the representation formula for linear wave equation, the solution $\phi$ to the semilinear wave equation \eqref{eq:NLW:semi:2d} verifies
\begin{equation}
\label{eq:large:p:1}
  |\phi(t_0, x_0)|\les
  (1+t_0+r_0)^{-\frac{1}{2}}+\int_0^{t_0}\int_{|x_0-y|<t_0-s}\frac{|\phi|^p(s, y)}{\sqrt{(t_0-s)^2-|x_0-y|^2}}dy ds.
\end{equation}Using polar coordinate $ y=r(\cos\theta, \sin\theta)$ and the definition of $\phi_*$ we have
\begin{align*}
&\int_{ |x_0-y|<t_0-s } \frac{|\phi|^p(s, y)}{\sqrt{(t_0-s)^2-|x_0-y|^2}}dy \\
\leq& \int_{r>0}\int_{\theta\in I(r)}\frac{\phi_*^p(s, r)}{\sqrt{(t_0-s)^2-|x_0-r(\cos\theta, \sin\theta)|^2}}rd\theta dr,
\end{align*}
where (for fixed $x_0,t_0,s$)
\begin{align*}
& I(r)=\{\theta\in[-\pi,\pi): (t_0-s)>|x_0-r(\cos\theta, \sin\theta)|\}.
\end{align*}
Note that
\begin{align*}
|x_0-r(\cos\theta, \sin\theta)|^2=r_0^2+r^2 -2rr_0 \cos(\theta_0-\theta),\quad x_0=r_0(\cos\theta_0, \sin\theta_0).
\end{align*}Thus we can estimate that
\begin{align*}
&\int_{\theta\in I(r)}\frac{1}{\sqrt{(t_0-s)^2-|x_0-r(\cos\theta, \sin\theta)|^2}}d\theta\\
=& \int_{\theta\in I(r)}\frac{1}{\sqrt{(t_0-s)^2-r_0^2-r^2 +2rr_0 \cos(\theta_0-\theta)}}d\theta\\
=& \int_{\theta\in \widetilde{I}(r)}\frac{1}{\sqrt{(t_0-s)^2-r_0^2-r^2 +2rr_0 \cos\theta}}d\theta\\
=&\frac{F(2^{-1}r^{-1}r_0^{-1}((t_0-s)^2-r_0^2-r^2))}{\sqrt{2rr_0 }}.
\end{align*}
Here (for fixed $x_0,t_0,s$) the function $F$ is defined in Lemma \ref{lem:log2} and
\begin{align*}
& \widetilde{I}(r)=\{\theta\in[-\pi,\pi):(t_0-s)^2-r_0^2-r^2 +2rr_0 \cos\theta>0\}.
\end{align*}
By using Lemma \ref{lem:log2}, we then can show that
\begin{align*}
&\int_{|x_0-y|<t_0-s}\frac{|\phi|^p(s, y)}{\sqrt{(t_0-s)^2-|x_0-y|^2}}dy \\
\leq& \int_{r>0}\frac{\phi_*^p(s, r)}{\sqrt{2rr_0}} F( 2^{-1}r^{-1}r_0^{-1} ((t_0-s)^2-r_0^2-r^2) ) r dr\\
\leq& C\int_{r>0}\frac{\phi_*^p(s, r)}{\sqrt{2rr_0}}(1+\ln(1+2rr_0 |(t_0-s)^2-(r_0+r)^2|^{-1}) r dr\\
\leq& C\int_{r>0}\frac{\phi_*^p(s, r)}{\sqrt{2rr_0}}(1+\ln(1+2r/|t_0-s-r_0-r|))r dr.
\end{align*}
Here in the last step we used (for $t_0>s>0$)
\begin{align*}
& |(t_0-s)^2-(r_0+r)^2|=|t_0-s-r_0-r||t_0-s+r_0+r|\geq |t_0-s-r_0-r|r_0.
\end{align*}
Now using Lemma \ref{lem:log1} with $g(r)=r^{\frac{1}{2}}\phi_*^p(s, r)$, $x=t_0-s-r_0 $ as well as estimate \eqref{eq:rp} , we conclude that
\begin{equation}
\label{eq:p2}
\begin{split}
&\int_{|x_0-y|<t_0-s}\frac{|\phi|^p(s, y)}{\sqrt{(t_0-s)^2-|x_0-y|^2}}dy \\
&\les  \int_{r>0}\frac{\phi_*^p(s, r)}{\sqrt{2rr_0}}(1+\ln(1+s+r))r dr+ r_0^{-\frac{1}{2}}(1+s)^{-\frac{3}{2}}.
\end{split}
\end{equation}
Note that for positive constant $\gamma>0$, the function $z^{-\gamma}(1+\gamma\ln z)$ of $z$ is decreasing for $z>1$. Hence  for $s>0,\quad r>0$ and $ \gamma=p/2-1$, we can show that
\begin{align*}
r^{\frac{1}{2}}(1+s)^{\frac{p-2}{2}}(1+\gamma\ln(1+s+r))&\leq r^{\frac{1}{2}}(1+s+r)^{\frac{p-2}{2}}(1+\gamma\ln(1+s))\\&\leq (1+s+r)^{\frac{p-1}{2}}(1+\gamma\ln(1+s)),
\end{align*}
which together with estimates \eqref{eq:p2}, \eqref{eq:trp} leads to
\begin{align*}
&\int_{|x_0-y|<t_0-s}\frac{|\phi|^p(s, y)}{\sqrt{(t_0-s)^2-|x_0-y|^2}}dy \\
&\les (1+s)^{-\frac{p-2}{2}}r_0^{-\frac{1}{2}}(\ln(1+s)+1)\int_{r>0}(1+s+r)^{\frac{p-1}{2}}\phi_*^p(s, r)dr+r_0^{-\frac{1}{2}}(1+s)^{-\frac{3}{2}}\\
&\les (1+s)^{-\frac{p-2}{2}}r_0^{-\frac{1}{2}}(\ln(1+s)+1)(1+s)^{-\frac{p-3}{4}}|f(s)|^{\frac{p-3}{2}}+r_0^{-\frac{1}{2}}(1+s)^{-\frac{3}{2}}\\
&\les r_0^{-\frac{1}{2}}(\ln(1+s)+1)(1+s)^{\frac{7-3p}{4}}|f(s)|^{\frac{p-3}{2}}+r_0^{-\frac{1}{2}}(1+s)^{-\frac{3}{2}}.
\end{align*}
Multiply both side of \eqref{eq:large:p:1} with $r_0^{\frac{1}{2}}$, we have
\begin{align*}
|\phi(t_0, x_0)|r_0^{\frac{1}{2}}\les 1 +\int_0^{t_0}  ((\ln(1+s)+1)(1+s)^{\frac{7-3p}{4}}|f(s)|^{\frac{p-3}{2}}+(1+s)^{-\frac{3}{2}})ds.
\end{align*}
On the other hand if $r_0=|x_0|\leq \frac{1}{2}\max(1,t_0)$, then by Proposition \ref{prop:ptdecay:allp} we have the improved decay estimates
\begin{align*}
|\phi(t_0, x_0)|\les
(1+t_0)^{-\frac{p-1}{4}+\ep}\les (1+t_0)^{-\frac{1}{2}}.
\end{align*}
Hence taking supreme in terms of $x_0$ and in view of the definition for $f(s),\ f_*(s)$, we derive that
\begin{align*}
f(t_0)\leq f_*(t_0)\les 1+\int_0^{t_0}  (\ln(1+s)+1)(1+s)^{\frac{7-3p}{4}}|f(s)|^{\frac{p-3}{2}}ds.
\end{align*}
For the case when $\frac{11}{3}<p<5$, we in particular have that
$$\frac{3p-7}{4}>1>\frac{p-3}{2}>0, $$
which together with the previous estimate leads to the conclusion that
$$
f_*(t_0)\les 1.
$$
The Proposition then follows by our convention that the implicit constant relies only on $p$ and the initial data $\mathcal{E}_{0, 2}$, $\mathcal{E}_{1, 0}$.
\end{proof}

\bibliography{shiwu}{}

\begin{thebibliography}{10}

\bibitem{Brezis80:LSobolev}
H.~Br\'{e}zis and T.~Gallouet.
\newblock Nonlinear {S}chr\"{o}dinger evolution equations.
\newblock {\em Nonlinear Anal.}, 4(4):677--681, 1980.

\bibitem{Brezis80:LSobolev:g}
H.~Br\'{e}zis and S.~Wainger.
\newblock A note on limiting cases of {S}obolev embeddings and convolution
  inequalities.
\newblock {\em Comm. Partial Differential Equations}, 5(7):773--789, 1980.

\bibitem{newapp}
M.~Dafermos and I.~Rodnianski.
\newblock A new physical-space approach to decay for the wave equation with
  applications to black hole spacetimes.
\newblock In {\em X{VI}th {I}nternational {C}ongress on {M}athematical
  {P}hysics}, pages 421--432. World Sci. Publ., Hackensack, NJ, 2010.

\bibitem{Dodson:NLW:3d:p3}
B.~Dodson.
\newblock {Global well-posedness and scattering for the radial, defocusing,
  cubic nonlinear wave equation}.
\newblock 2018.
\newblock ar{X}iv:1809.08284.

\bibitem{Dodson:NLW:3d:p35}
B.~Dodson.
\newblock {Global well-posedness for the radial, defocusing, nonlinear wave
  equation for $3<p<5$}.
\newblock 2018.
\newblock ar{X}iv:1810.02879.

\bibitem{Moncrief1}
D.~Eardley and V.~Moncrief.
\newblock The global existence of {Y}ang-{M}ills-{H}iggs fields in
  {$4$}-dimensional {M}inkowski space. {I}. {L}ocal existence and smoothness
  properties.
\newblock {\em Comm. Math. Phys.}, 83(2):171--191, 1982.

\bibitem{velo85:global:sol:NLW}
J.~Ginibre and G.~Velo.
\newblock The global {C}auchy problem for the nonlinear {K}lein-{G}ordon
  equation.
\newblock {\em Math. Z.}, 189(4):487--505, 1985.

\bibitem{Velo87:decay:NLW}
J.~Ginibre and G.~Velo.
\newblock Conformal invariance and time decay for nonlinear wave equations.
  {I}, {II}.
\newblock {\em Ann. Inst. H. Poincar\'{e} Phys. Th\'{e}or.}, 47(3):221--261,
  263--276, 1987.

\bibitem{Pecher82:NLW:2d}
R.~Glassey and H.~Pecher.
\newblock Time decay for nonlinear wave equations in two space dimensions.
\newblock {\em Manuscripta Math.}, 38(3):387--400, 1982.

\bibitem{Hidano:scattering:NLW}
K.~Hidano.
\newblock Scattering problem for the nonlinear wave equation in the finite
  energy and conformal charge space.
\newblock {\em J. Funct. Anal.}, 187(2):274--307, 2001.

\bibitem{Hidano03:scattering:NLW:56D}
K.~Hidano.
\newblock Conformal conservation law, time decay and scattering for nonlinear
  wave equations.
\newblock {\em J. Anal. Math.}, 91:269--295, 2003.

\bibitem{Yang:NLW:1d}
G.~Luli, S.~Yang, and P.~Yu.
\newblock On one-dimension semi-linear wave equations with null conditions.
\newblock {\em Adv. Math.}, 329:174--188, 2018.

\bibitem{Pecher82:decay:3d}
H.~Pecher.
\newblock Decay of solutions of nonlinear wave equations in three space
  dimensions.
\newblock {\em J. Funct. Anal.}, 46(2):221--229, 1982.

\bibitem{Strauss:NLW:decay}
W.~Strauss.
\newblock Decay and asymptotics for {$ cmu=F(u)$}.
\newblock {\em J. Functional Analysis}, 2:409--457, 1968.

\bibitem{vonWahl72:decay:NLW:super}
W.~von Wahl.
\newblock Some decay-estimates for nonlinear wave equations.
\newblock {\em J. Functional Analysis}, 9:490--495, 1972.

\bibitem{yang:scattering:NLW}
S.~Yang.
\newblock {Global behaviors for defocusing semilinear wave equations}.
\newblock 2019.
\newblock ar{X}iv:1908.00606.

\bibitem{yang:NLW:ptdecay:3D}
S.~Yang.
\newblock {Pointwise decay for semilinear wave equations in
  $\mathbb{R}^{3+1}$}.
\newblock 2019.
\newblock ar{X}iv:1908.00607.

\bibitem{yang:NLW:ptdecay:3D:smallp}
S.~Yang.
\newblock {Uniform bound for solutions of semilinear wave equations in
  $\mathbb{R}^{3+1}$}.
\newblock 2019.
\newblock ar{X}iv:1910.02230.

\end{thebibliography}
\bibliographystyle{plain}

\end{document}